\documentclass[a4paper, 12pt]{amsart}
\usepackage{amsfonts,amsmath,amssymb}
\usepackage{amssymb}
\usepackage{graphicx}
\usepackage{color}

\usepackage{ textcomp }
\mathchardef\emptyset="001F

\theoremstyle{plain}

\newtheorem{theorem}{Theorem}[section]
\newtheorem{lemma}[theorem]{Lemma}
\newtheorem{proposition}[theorem]{Proposition}
\newtheorem{corollary}[theorem]{Corollary}
\newtheorem{remark}[theorem]{Remark}
\newtheorem{definition}[theorem]{Definition}

\numberwithin{equation}{section}

\newcommand{\eps}{\epsilon}

\newcommand{\R}{{\mathbb R}}

\newcommand{\Z}{{\mathbb Z}}
\newcommand{\N}{{\mathbb N}}

\newcommand{\ZZ}{{\mathbb Z}}

\newcommand{\NN}{{\mathbb N}}

\newcommand{{\EE}}{{\mathbb E}}

\newcommand{\RR}{{\mathbb R}}

\newcommand{\T}{{\mathbb T}}

\newcommand{\TT}{{\mathbb T}}

\newcommand{\dive}{{\rm div}}

\renewcommand{\div}{\hbox{{\rm div}}}
\newcommand{\grad}{\hbox{{\rm grad}}}

\newcommand{\ut}{\widetilde{u}_{\e}}

\newcommand{\bfe}{{\bf e}}

\newcommand{\caE}{{\mathcal E}}

\newcommand{\caM}{{\mathcal M}}
\newcommand{\caO}{{\mathcal O}}

\newcommand{\caA}{{\mathcal A}}

\newcommand{\caF}{{\mathcal F}}

\def\2s{\stackrel{\rm 2s}{\rightharpoonup}}

\def\ut{{\underline t}}

\def\uv{{\underline v}}

\def\urho{{\underline \rho}}

\title[]
{Crime Pays\,;\
Homogenized Wave Equations for Long Times}

\author[Gr\'egoire Allaire]{Gr\'egoire Allaire$^{1}$}
\author[Agnes Lamacz]{Agnes Lamacz$^{2}$}
\author[Jeffrey Rauch]{Jeffrey Rauch$^{3}$}

\begin{document}
\baselineskip3.15ex
\vskip .3truecm

\begin{abstract}  This article examines the accuracy for large times of  asymptotic
expansions from periodic homogenization of wave equations.
As usual, $\eps$ denotes the small period of the 
coefficients in the wave equation. 
We first
prove that the standard two scale asymptotic expansion 
provides an accurate approximation of the exact solution for  times $t$ 
of order $\eps^{-2+\delta}$ for any $\delta>0$. 
Second, for longer times, we show that a different algorithm, 
that is called criminal because it mixes different powers of $\eps$, 
yields an approximation 
of the exact solution with error $O(\eps^N)$  for times 
$\eps^{-N}$ with $N$ as  large as one likes. 
The criminal algorithm involves high order homogenized equations 
that, 
in the context of the wave equation, were first proposed by Santosa and Symes and 
analyzed by Lamacz. 
The high order homogenized equations yield dispersive corrections
for  moderate wave numbers.
We give a systematic analysis for all time scales and 
all high order corrective  terms.

\small{
\vskip.2truecm
\noindent  {\bf 2010 Mathematics Subject Classification:}

}
\end{abstract}

\maketitle
{
\small
\noindent
$^1$ Centre de Math\'ematiques Appliqu\'ees, \'Ecole Polytechnique, CNRS,
Universit\'e Paris-Saclay, 91128 Palaiseau, France.\\
Email: gregoire.allaire@polytechnique.fr\\
\noindent
\noindent
$^2$ Department of Mathematics, University of Duisburg-Essen,
45127 Essen, Germany.\\
Email: agnes.lamacz@uni-due.de\\
\noindent
\noindent
$^3$ Department of Mathematics, University of Michigan,
Ann Arbor 48109 MI, USA.\\
Email: rauch@umich.edu
}

\maketitle

\section{Introduction}

\subsection{Traditional homogenization, secular growth, and long times}

This paper studies the long time behavior of the
wave equation in an infinite periodic medium,
\begin{equation}
\label{eq:homprob}
\rho(x/\eps) \, \partial^2_t u^\eps \, - \, \big(\dive \, a(x/\eps)\, \grad\big)\, u^\eps 
\ =\
 f(t,x),
\qquad
u^\eps=f=0\ \ {\rm for}\ \ t<0\,,
\end{equation}
where $\rho, a$ are periodic functions.
The unknown $u^\eps$ is real valued as is $\rho\in L^\infty(\TT^d)$
while $a\in L^\infty(\TT^d)$ has values that are real symmetric matrices ($\TT^d$ is the unit torus).
The coefficients are postive definite in the sense that there is a 
constant $m_1>0$ so that for all $\xi\in \RR^d$,
$$
a(y)\xi\cdot \xi\ \ge\ m_1|\xi|^2,
\qquad
{\rm and}
\qquad
\rho(y)\ \ge \ m_1,
\qquad
{\rm for\ a.e.\ \ } y\in\TT^d \,.
$$
The source term $f$ is smooth with 
$\partial_{t,x}^\alpha f \in L^2(\RR^{1+d})$ for all $\alpha\in \NN^{d+1}$ 
and is supported in  $\{0\le t\le 1\}$ until Section \ref{sec.polynom}.

The motivation comes from the articles, in chronological order,
\cite{santosa}, \cite{Lam}, \cite{DLS}, \cite{AP2}, \cite{BG}
that describe the behavior of solutions on the very long 
time scale $t\sim 1/\eps^2$ and beyond (see also
the engineering literature, including \cite{andrianov}, \cite{fish}, 
and the numerical literature \cite{AGS}, \cite{AP}, \cite{AY}). 
The descriptions for these time scales use modifications 
of the traditional two scale homogenization {\it ansatz}.

The traditional {\it ansatz} is 
$U(\eps,t,x,x/\eps)$, where
\begin{equation}
\label{eq:utaylor}
U(\eps, t,x,y) \ \sim\
\sum_{n=0}^\infty
\,
\eps^n\,
u_n(t,x,y)\,,
\qquad
u_n(t,x,y)\quad {\rm periodic\ in \ } y\,.
\end{equation}
The right hand side  is a formal power series in $\eps$.  
No convergence is expected (see Appendix
\ref{sec:maxgrowth}). The sign $\sim$ represents equality  
in the sense of formal power series. 
The coefficient functions $u_n$ belong to the space of 
smooth functions of $(t,x,y)\in\RR\times\RR^d\times\TT^d$ 
supported in $t\ge 0$  that are periodic in $y$.

Section \ref{sec:2minusdelta} proves that the traditional
construction \eqref{eq:utaylor} yields a good approximations
on time intervals $0\le t\le C\,\eps^{-2+\delta}$
with $C$ as large and $0<\delta$ as small as one likes.
The classical approach \cite{BP}, \cite{BLP}, \cite{BFM}, \cite{JKO}, \cite{SP} 
proves that the ansatz \eqref{eq:utaylor} is a good
approximation on bounded time intervals.
It was first observed
by  Santosa and Symes in \cite{santosa}, and then proved in \cite{Lam} 
(see also \cite{DLS}, \cite{AP2}, \cite{BG}), that a different ansatz
that we call  {\it  criminal}, yields a good approximation for 
times of order $\eps^{-2}$.
In the elliptic setting, this criminal ansatz was first proposed 
by Bakhvalov and Panasenko \cite{BP}.

 To analyze the two scale ansatz \eqref{eq:utaylor}, each profile $u_n$ is written 
as the sum of its non oscillating contribution $\pi u_n$ and
its oscillating part $(I-\pi)u_n$, 
defined as
\footnote{This decomposition of the two scale hierarchy emphasizing the
projector $\pi$ follows modern developments in hyperbolic geometric optics,
see \cite{rauch}.} 
$$
u_n(t,x,y) \ =\ \pi u_n + (I-\pi)u_n,
\qquad
(\pi u_n )(t,x) \ :=\ 
\frac1{|\TT^d|}   \,
\int_{\TT^d} u_n(t,x,y)\ dy.$$
Introduce the traditional second order partial differential operators
\begin{equation}
\label{eq:operators}
\begin{aligned}
 \caA_{yy} & := 
 \dive_y \, a(y)\, \grad_y,
 \\  
  \caA_{xx} &:= 
 \dive_x \, a(y)\, \grad_x,
 \\ 
 \caA_{xy} &:= 
 \dive_x  \, a(y)\, \grad_y
+
 \dive_y \, a(y)\, \grad_x
 \,.
\end{aligned}
\end{equation}
Then
\begin{equation*}
\Big[
\rho(x/\eps) \, \partial^2_t \ -\ \dive\, a(x/\eps)\grad\Big]
U(\eps,t,x,x/\eps)
\ \sim\
 W(\eps,t,x,x/\eps)\,.
\end{equation*}
where $W$ is the formal Laurent series in $\eps$ computed from
\begin{equation}
\label{eq:defW}
 \bigg[
   \rho(y) \partial_t^2
    -
  \frac{1}
  {\eps^2}
  \mathcal{A}_{yy}
   - 
   \frac{1}
  {\eps}
  \mathcal{A}_{xy}
  - 
  \mathcal{A}_{xx}
  \bigg]
  U(\eps,t,x,y)
 \sim W(\eps,t,x,y)
 :=
  \sum_{n=-2}^\infty
 \eps^n \,w_n(t,x,y).
\end{equation}

In all constructions below (the classical one, as well as the new one 
proposed in this paper) the $u_n$ are chosen so that up to some precision, 
one has $W=f$. Since the source term $f(t,x)$ does not depend on $y$, it 
has no oscillating part, $(I-\pi)f=0$, and thus it is natural to seek 
the $u_n$ so that $(I-\pi)w_n=0$. The formal power series $U$, satisfying (\ref{eq:defW}), 
for which $(I-\pi)w_n=0$ have a very rigid structure that steers our analysis. 
For $k\geq1$, we introduce differential operators (see Definition \ref{def.chi-k} for details)
\begin{equation}
\label{eq:chi-k}
\chi_k(y,\partial_{t},\partial_x)\ =\ 
\sum_{\alpha\in \NN^{d+1}, \, |\alpha|=k}
c_{\alpha,k}(y)\,\partial_{t,x}^\alpha.
\end{equation}
The coefficients $c_{\alpha,k}$ are solutions of periodic cell problems. 
The coefficients of the pure $x$ derivatives in (\ref{eq:chi-k}) are the classical
$k^{\rm th}$ order correctors in elliptic homogenization \cite{BP}, \cite{BLP}, \cite{SP}.
Theorem \ref{thm:osc} proves that the ansatz $U$, which yield profiles $w_n$ 
satisfying $(I-\pi)w_n=0$, 
are so that  $u_n$ satisfy,
\begin{equation}
\label{eq:radwanska}
\forall n\ge 0,\qquad
 (I-\pi) u_{n}  \ = \ \sum_{k=1}^{n} \chi_k(y,\partial_{t},\partial_x) \  \pi u_{n-k}\,.
\end{equation}
The oscillating part  is given in terms of the 
non-oscillating parts of lower order.

The second structural identity concerning the formal series $U$ 
satisfying $(I-\pi)w_n=0$ is a formula for $\pi w_n$ that
involves 
homogenized differential operators $a_k^*(\partial_t,\partial_x)$ 
with constant coefficients.  
The operator $a_2^*$ is the standard homogenized wave operator 
\cite{BP}, \cite{BLP}, \cite{BFM}, \cite{JKO}, \cite{SP}. 
For $k\geq3$, the $a_k^*$'s are called high order homogenized operators \cite{BP}. 
By establishing a combinatorial formula for the $a_{k}^*$'s, 
 Theorem \ref{thm:recur2}
 proves that the odd order homogenized operators vanish, 
$a_{2n+1}^*=0$. 
This is a classical result for $a_1^*=0$ and $a_3^*=0$ (see e.g. 
\cite{ABV} and references therein).
It was already known for all odd orders in the 
elliptic case \cite{SmCh}.
Theorem \ref{thm:recur1} proves that
\begin{equation}
\label{eq:mildred}
\forall n,\quad
(I-\pi)w_n=0, \ 
\qquad
\Longleftrightarrow
\qquad
\forall n,\quad
\pi w_n \ = \ \sum_{0\le 2j\le n}\,
a^*_{2j+2}(\partial_t,\partial_x)\ \pi  u_{n-2j}  \,.
\end{equation}
Equation \eqref{eq:mildred} expresses $\pi w_n$ in terms 
of $\pi u_m$ with $m\le n$ {\it and having 
the same parity as $n$.}  
This is  {\it   the leap frog structure}.   
 
The traditional algorithm \cite{BP}, \cite{BLP}, \cite{SP} sets $W=f$.
Equivalently,
$w_0 = \pi w_0=f$ and $\pi w_n=0$ for $n\ge 1$.  
The first yields the homogenized wave equation 
$a_2^*(\partial_t,\partial_x) \pi u_0 =f$ whose solution $u_0=\pi u_0$
has energy independent of $t$ for $t$ beyond the support of $f$. 
The leading profile $\pi u_0$ does not grow with time. 
The next equations, $\pi w_n=0$ for $n\ge 1$, lead to 
 wave equations for each $\pi u_n$ with a source term 
given in terms of the preceding  $\pi u_m$, $m<n$, with the same
parity as $n$. One finds $\pi u_n=0$ for $n$ odd.  
One can  quickly assess the rate of growth of the profiles $u_{2n}$ in time.  
This is called {\it secular growth} (see Section \ref{sec:secular}).     
For $u_2$ one has 
$a_2^*(\partial_t,\partial_x)\pi u_2 + a_4^*(\partial_t,\partial_x)\pi u_0 =0$, 
a wave equation for $\pi u_2$
with source that does not grow in time. Therefore $u_2$ cannot
grow faster than $t$. Continuing one finds that  $\pi u_4$ grows
no faster than $t^2$ and $\pi u_{2k}$ no faster than $t^k$.
The leap frog structure shows that $\pi u_n$ grows no faster
than $t^{n/2}$. Without the leap frog structure one would have
found $t^n$. The $(2n)^{\rm th}$ term in the ansatz \eqref{eq:utaylor} 
is of size $\eps^{2n}\, t^n$.
For times $t\sim 1/\eps^2$ the higher order terms can no longer
be understood as correction terms.  The slow secular growth from the leap frog
structure explains
why $t\sim 1/\eps^2$ is a critical time scale for the 
traditional expansion.

The secular growth estimate implies
Theorem \ref{thm:errorest} asserting 
that {\it for any $N,\delta>0$ choosing a sufficiently large number of 
terms in the traditional
ansatz \eqref{eq:utaylor}
guarantees that  the error is $O(\eps^N)$
for times
$t\le 1/\eps^{2-\delta}$}.

Appendix \ref{sec:beyond}
contains an example showing that the classical approximation is
not accurate for times $t\sim 1/\eps^{2+\delta}$ for any 
$\delta>0$.

\subsection{Asymptotic crimes and longer times}
\label{sec.crimintro}
To find approximate solutions for longer times 
we abandon the classical {\it ansatz} that requires
that $w_0=f$ and $w_n=0$ for $n\ne 1$.  In the residual
we do not set the coefficients of $\eps^n$ equal to zero.
This is  an {\it asymptotic crime}.  In addition to the 
motivating examples from homogenization theory,
asymptotic crimes have a long history in fluid dynamics
and geometric optics, see \cite{lannes}.
In order that the computations retain much of the structure
from the traditional {\it ansatz} we  demand that
$(1-\pi)w_n=0$ for all $n$.    
That yields \eqref{eq:radwanska} and \eqref{eq:mildred} and, in particular, 
the leading term is non oscillating, $v_0=\pi v_0$.
To emphasize the fact that the new profiles are 
not the  same as the  old ones we call them
$v_n$ and set
$$
V(\eps, t,x,y) 
\ \sim\
\sum_{n=0}^\infty
\,
\eps^n\,
v_n\,.
$$
Then \eqref{eq:mildred} implies that
the discrepancy $W-f=\pi(W-f)$ is 
equal to the sum of the lines,
\begin{equation}
\label{eq:hierarchy2.3-0}
\begin{aligned}
&\eps^0\big[a^*_2(\partial_t, \partial_x)\pi v_0\ -\ f\big]\ +\   
\cr
&\eps^1\big[a^*_2(\partial_t, \partial_x) \pi v_1 \big]\ +\ 
\cr
&\eps^2\big[a^*_2(\partial_t, \partial_x) \pi v_2  \ +\  a_4^*(\partial_t, \partial_x)\pi v_0\big]\ +\ 
\cr
&\eps^3\big[a^*_2(\partial_t, \partial_x) \pi v_3  \ +\  a_4^*(\partial_t, \partial_x)\pi v_1\big]\ +\ 
\cr
&\eps^4\big[a^*_2(\partial_t, \partial_x) \pi v_4  \  +\   a_4^*(\partial_t, \partial_x)\pi v_2 \ +\  a_6^*(\partial_t,\partial_x) \pi v_0\big]\ +
\cr
&\eps^5\big[a^*_2(\partial_t, \partial_x) \pi v_5  \  +\   a_4^*(\partial_t, \partial_x)\pi v_3 \ +\  a_6^*(\partial_t,\partial_x) \pi v_1\big]\ +\ \cdots\ .
\end{aligned}
\end{equation}
The problems of secular growth came from setting all the rows
equal to zero. 
That yields equations for the corrector terms
that have the preceding profiles as sources. The criminal
strategy requires only that the sum
of the  lines vanishes. That can be achieved setting
$\pi v_n=0$ for all $n>0$ and demanding that
\begin{equation}
\label{eq:julie}
\Big[
a_2^*(\partial_{t,x}) + \eps^2  a_4^*(\partial_{t,x}) + \eps^4 a_6^*(\partial_{t,x})
 +
\cdots
\Big] 
v_0  \ =\  f\,,
\qquad
v_0=0 \quad
{\rm for}
\quad
t<0\,.
\end{equation}
The coefficients in Equation \eqref{eq:julie} 
depend on $\eps$. To solve 
this equation with accuracy $O(\eps^N)$,
$v_0$ 
must depend on 
$\eps$,  $v_0=v_0(\eps, t,x)$.  
Including oscillatory correction terms, the approximation takes the new form
\begin{equation}
\label{eq:crimeansatz}
V(\eps, t,x,y)\ \sim\
\sum_{n=0}^{\infty}
\,
\eps^n\,
v_n(\eps,t,x,y)\,,
\qquad
v_n(\eps,t,x,y)\quad {\rm periodic\ in \ } y\,,
\end{equation}
where each profile depends on $\eps$.
The series is treated as a formal series in $\eps$ whose
coefficients are functions of several variables including $\eps$. 
The summand $\eps^nv_n$ is viewed as a term in $\eps^n$ 
although $v_n$  depends on $\eps$.

To construct the profile $v_0$,
\eqref{eq:julie}  is modified in several
ways.
  The first difficulty is that the terms
$\eps^{2j-2}\,a_{2j}^*(\partial_t,\partial_x)$ are typically
of order  $2j$ in $\partial_t$.   The more terms one
keeps the higher order is the equation in $\partial_t$. 
    The truncated operators
usually  define ill posed initial value problems.
The first thing that we do is perform a normal form transformation
that converts the operators $a_{2j}^*$ with $j\geq 2$ to operators in $\partial_x$ only.  
The normal form removes all the time derivatives other than those
in $a_2^*$.
In Section \ref{sec:elimination} it is proved that there are uniquely determined
homogeneous operators $R_{2j}(\partial_{t,x})$ and $\widetilde a_{2j}(\partial_x)$
of  degree $2j$
so that as formal power series in $\partial_t,\partial_x$,
\begin{equation}
\label{eq:R2j}
\Big[
1+\sum_{j=1}^\infty R_{2j}\big(\partial_{t,x}\big)
\Big]
\Big[
\sum_{j=1}^\infty a_{2j}^*\big(\partial_{t,x}\big)
\Big]
\ =\ 
a_2^*\big(\partial_{t,x}\big) \ + \
\sum_{j=2}^\infty\,
\widetilde a_{2j}\big(\partial_x\big)
\,.
\end{equation}
The operators $R_{2j}$ and $\widetilde a_{2j}$ are 
computable by a rapid recursive algorithm.
This step has no analogue in the elliptic context.
Multiplying \eqref{eq:julie} by $1+\sum_{j=1}^\infty R_{2j}\big(\eps\partial_{t,x}\big)$
yields the equivalent equation
\begin{equation}
\label{eq:disp5}
\Big[
a^*_2(\partial_{t,x} )
 \ +\
 \sum_{j=2}^{\infty}
 \eps^{2j-2}\
\widetilde a_{2j}(\partial_x)
 \Big]
 v_0(\eps,t,x) \ =\ 
 \Big[
 1+\sum_{j= 1}^\infty \eps^{2j} R_{2j}\big(\partial_{t,x}\big)
 \Big]
 f\,.
\end{equation}
One does not need an exact solution.   
The sums in 
 \eqref{eq:disp5}  are first truncated to finite sums.
  The corresponding 
  equation depends on the number of terms retained and 
  the unknown function is denoted
 $\uv_0^k$. The truncated equation of order $k$ is,
 with  $R^k:=\sum_{j=1}^kR_{2j}$,
\begin{equation}
\label{eq:disp}
\Big[
a^*_2(\partial_{t, x})
 \ +\
 \sum_{j=2}^{k+1}
 \eps^{2j-2}\
 \widetilde a_{2j}(\partial_x)
 \Big]
 \uv_0^k(\eps,t,x) =
 \Big[
 1+
 R^k\big(\eps\partial_{t,x}\big)\Big]f,
 \quad
 \uv_0^k=0 \ \ {\rm for}\ \ t<0\,.
\end{equation}

The initial value problem 
\eqref{eq:disp}  is usually ill  posed so does not define
a profile $\uv_0^k$.
For example, it is known \cite{COV2}, \cite{Lam}, \cite{ABV} 
that, at least when $\rho(y)$ is constant, the operator $\widetilde a_4$ 
has the wrong sign so that 
\eqref{eq:disp} is ill posed for $k=2$. 
Surprisingly, that is not a fatal flaw.

The idea to overcome this obstacle is the following. 
The correctors $\eps^{2j-2}\widetilde a_{2j}(\partial_x)$ added to
$a_2^*(\partial_t, \partial_x)$ are only small compared to $a_2^*$ when
applied to functions whose Fourier transform is supported where 
$\eps \xi$ is small ($\xi$ being the Fourier variable).
The idea is to filter the source term.
Choose $\psi_1\in C^\infty_0(\RR^d)$  equal to 1 on a neighborhood
of the origin.  Choose  $0<\alpha<1$.
The operator $\psi_1(\eps^\alpha D)$ is the Fourier multiplier 
 $g\mapsto \caF^{-1}\psi_1(\eps^\alpha \xi)\caF g$.  Equivalently $D:=(1/i)\partial_x$.
The filtered equation  is
\begin{equation}
\label{eq:filtered}
\Big[
a^*_2(\partial_{t,x})
 \ +\
 \sum_{j=2}^{k+1}
 \eps^{2j-2}\
 \widetilde a_{2j}(\partial_{x})
 \Big]
 v_0^k \, =\, 
 \psi_1(\eps^\alpha D)\big(1+R^{k}(\eps \partial_{t,x} ) \big)f,
 \quad
 v_0^k=0 \  {\rm for}\ t<0.
\end{equation}
The right hand side has Fourier transform supported in 
$|\xi|\leq C \eps^{-\alpha} \ll 1/\eps$.

Equation \eqref{eq:filtered} is the one that is solved to determine a profile $v_0^k$.
The filtered equation
\eqref{eq:filtered}  has  a unique tempered solution.
That solution has
spatial Fourier transform  supported in $\eps^{-\alpha}\,{\rm supp}\,\psi_1$.
Energy bounds  like those for $a_2^*$ are proved in Section \ref{sec:stability}.
The operator on the left in \eqref{eq:disp} is the sum of the 
homogenized operator and hopefully small higher order terms. The higher order terms are sometimes thought of 
as dispersive correctors. This is at least the original interpretation of $\widetilde a_4$ in \cite{santosa}.  
The next definition  summarizes  the recipe for the approximate solution.

\begin{definition} [\bf Criminal approximation]
\label{def:criminal}   Fix the choice of 
$\psi_1\in C_0^\infty(\R^d)$, $0<\alpha<1$,  and $0\le k\in \NN$.
Define profiles $v_n^k(\eps, t, x, y)$ for $0\le n\le 2k+2$ as follows.

$\bullet$  Nonoscillatory parts.  For $ 1\le n\le 2k+2$,
$\pi v_n^k=0$.   For $n=0$, $\pi v_0^k$ is  the unique tempered
solution of  the high order homogenized equation \eqref{eq:filtered}.

$\bullet$  Oscillatory parts.   For $1\le n\le 2k+2$,  
$(I-\pi) v_{n}^k   = \chi_n\,\pi v_0^k$, where $\chi_n$ is defined by 
\eqref{eq:chi-k}, and $(I-\pi) v_{0}^k   =0$. 
Equivalently, $v_{0}^k=\pi v_{0}^k$. 
 
Define 
\begin{equation}
\label{eq:defVk}
V^k(\eps,t,x,y)\ :=\ \sum_{n=0}^{2k+2}\,
 \eps^n\, v_n^k(\eps,t,x,y) 
=\Big(I + \sum_{n=1}^{2k+2}\eps^n\chi_n(y,\partial_t,\partial_x)\Big)
v_0^k(\eps,t,x)\, .
\end{equation}
The criminal approximate solution is $V^k(\eps, t,x,x/\eps)$.
\end{definition}

The main result of the present paper is the following approximation theorem.

\begin{theorem}[\bf Criminal error] 
\label{thm:criminal}
Suppose that  $u^\eps$ is the 
exact solution  of \eqref{eq:homprob}
and $V^k$ is
given by \eqref{eq:defVk}.
For each $k\ge2$ there are positive  constants $C,\eps_0$
 so that for $0<\eps <\eps_0$
 and $t\ge 0$,
 the error in energy satisfies
\begin{equation*}
\big\|
\nabla_{t,x} \big(u^\eps(t,x) \ -\ V^k(\eps, t,x,x/\eps)\big)
\big\|_{L^2(\RR^d_x)}
\ \le \ C\, \eps^{2k+1}\, \langle t\rangle\,,
\quad
{\rm with} \quad \langle t\rangle := \sqrt{1+t^2}\,.
\end{equation*}
\end{theorem}

\begin{remark}
{\bf i.}   If one wants the error to decrease as $\eps^{N_1}$ on time intervals $0\leq t\leq C/\eps^{N_2}$, 
it suffices to choose $k$ so that 
$N_1+N_2 \le 2k+1$.
  
{\bf ii.}  The initial value problem defining $v^k_0$ has constant
coefficients. Its spatial Fourier Transform is given by an explicit
formula.    A spectrally accurate approximate solution is 
computable by FFT.
\end{remark}

Writing $u=\int_0^t u_t\,dt$ yields the following  corollary. 

\begin{corollary}
With the assumptions and notations in Theorem \ref{thm:criminal}, 
the error measured in $L^2(\RR^d)$ satisfies
\begin{equation*}
\big\|
u^\eps(t) \ -\ V^k(\eps, t,x,x/\eps)
\big\|_{L^2(\RR^d)}
\ \le \ C\, \eps^{2k+1}\, \langle t\rangle^2\,.
\end{equation*}
\end{corollary}

A more subtle corollary is that the oscillating part of the approximate solution is not necessary 
for the long time asymptotics if one is content with an error of the order of $\eps$.

\begin{corollary}
\label{cor.criminal}
With the assumptions and notations in Theorem \ref{thm:criminal}, 
the error from the leading term $v_0^k$ satisfies
\begin{equation*}
\big\|
u^\eps(t) \ -\ v_0^k(\eps, t,x)
\big\|_{L^2(\RR^d)}
\ \le \ C\, \left( \eps + \eps^{2k+1}\, \langle t\rangle^2 \right) \,.
\end{equation*}
\end{corollary}

Corollary \ref{cor.criminal} shows that for $N$ as large as one likes, 
if one takes $k\geq N$, then uniformly on 
$0\leq t\leq C/\eps^N$ the $L^2(\RR^d)$-error is smaller than $\eps$ 
using only the leading nonoscillatory term $v_0^k(\eps, t,x)$
in the approximate solution. Corollary \ref{cor.criminal} was 
proved in \cite{BG} in a more general context (almost 
periodic or random coefficients) with a  proof based 
on Bloch waves. Theorem \ref{thm:criminal} 
improves previous results since, not only 
the approximation error is valid for times as large as one wants,
but the error is as high order in $\eps$  as one wants.
Our results improve those of \cite{Lam}, \cite{DLS} 
which were restricted to times of order $1/\eps^2$ with an error 
of order $\eps$. The first paper \cite{Lam} relies on two scale 
asymptotic expansions, while the second one \cite{DLS} uses Bloch 
waves.

The  previous works \cite{Lam}, \cite{DLS}, \cite{BG} 
considered \eqref{eq:homprob} with $f=0$ and nonvanishing
initial data.
 In addition 
 in \cite{DLS} and \cite{BG}, $\partial_t u(0)=0$.
 One of the reasons that we can  push the 
analysis  further is that our choice simplifies some things.
We next expand a little on this choice. 
The solutions of 
\eqref{eq:homprob} with $f$ smooth in time with values in 
$L^2(\RR^d)$ and supported in a compact time interval
satisfy
\begin{equation}
\label{eq:noosc}
\forall j\in\NN,\quad
\sup_{0<\eps<1} \
\sup_{t\in \RR}
\
\big\|
\nabla_{t,x} 
\partial_t^j\,
u^\eps(t)
\big\|_{L^2(\RR^d)}
\ <\ 
\infty.
\end{equation}
The coefficients
vary on the small scale $\epsilon$
but the solutions do not
oscillate in time.   
For smooth solutions with $f=0$ and
 Cauchy data $u^\eps(0), \partial_t u^\eps(0)$
 the initial derivatives satisfy
 \begin{align*}
 \forall j\ge 0, \ \ell\in \{0,1\},\qquad
 \partial_t^{2j+\ell} u^\eps(0) &= 
 \bigg(
 \frac1{\rho(x/\eps)}
 \big(
 \div\, a(x/\eps)\, \grad\big)
 \bigg)^j \partial_t^\ell u^\eps(0)
\,.
 \end{align*}
These yield formulas for 
$
\nabla_{t,x} 
\partial_t^j\,
u^\eps\big|_{t=0}
\ :=\ H_j^\eps\big(u^\eps(0),\partial_t u^\eps(0)\big).
$
The initial data corresponding to solutions satisfying
\eqref{eq:noosc} are those so that 
\begin{equation}
\label{eq:awkward}
\forall j,\qquad
\sup_\eps \big\|H_j^\eps\big(u^\eps(0), \partial_t u^\eps(0) \big) \big\|_{L^2(\RR^d)}
\ <\
\infty\,.
\end{equation}
For the $f=0$ problem with Cauchy data satisfying \eqref{eq:awkward}
the approximation properties
for both classical and criminal strategies are as in our Theorems.
For the $f=0$ problem the accuracy of the approximation
is determined by how well the initial data can be appproximated
by data satisfying \eqref{eq:awkward}.

The condition \eqref{eq:awkward}  is awkward to use.
For example when $\rho,a$ are just $L^\infty(\TT^d)$ and at least
one of them is not constant
it is true but not immediately obvious that no family of initial
data that is independent  of $\eps$ can satisfy this condition.  
Without performing a nontrivial computation it is not
clear that there are initial data given by 
two scale expansions that satisfy this condition.
The solutions from traditional homogenization  with our source term
$f(t,x)$ viewed 
for $t$  beyond ${\rm supp}\,f$ show that there are
many such two scale data. 

Equation \eqref{eq:homprob} shows that 
solutions  that do not oscillate in time 
are important.  Their description via Cauchy data is 
awkward. 
We study the natural problem \eqref{eq:homprob}.

\subsection{Outline of the present paper}

In Section \ref{sec.twoscale} the classical two scale asymptotic expansion 
for wave equations
is analysed.
Theorem \ref{thm:recur2} proves that 
the odd order homogenized operators vanish.
Theorem \ref{thm:secular}  shows that secular growth of the  profiles 
is half as fast as one might expect.
Remark \ref{rem:BP1} is a first version of the criminal 
path.

Section \ref{sec:2minusdelta} studies the accuracy of the 
classical
expansion. Classical proofs show that for for bounded time
and any $N$ the error is  $O(\eps^N)$.
We  prove that taking more corrector terms one has
  $O(\eps^N)$ accuracy for times of order 
$\eps^{-2+\delta}$, for any $\delta>0$.

Section \ref{sec.crim} presents the details of  the derivation of 
the criminal asymptotic 
expansion and proves Theorem \ref{thm:criminal}.  

Section \ref{sec.polynom} shows that our results
for  sources $f$ compactly supported in time suffice, by a simple argument,
to treat sources that
 grow at most polynomially in time

Section \ref{sec.system}  discusses second order systems 
 including linear elasticity and Schr\"odinger's equation.  The place where
the argument is not automatic  is  the proof that the 
odd order homogenized operators vanish in the systems case
(Theorem \ref{thm:rec2}).

Appendix \ref{sec:maxgrowth} 
gives an example in dimension $d=1$ for which
the upper bound on the  the secular growth predicted in 
Section \ref{sec.twoscale} is attained. 
For the same example, 
 the classical two 
 scale asymptotic expansion \eqref{eq:utaylor} does not
yield a good approximation  for times  $t\sim 1/\eps^{2+\delta}$. 

Appendix \ref{sec.estimate} provides a classical a priori estimate for two scale oscillating 
functions.

Appendix \ref{sec.stability} proves that solutions of the  wave equation
have finite energy for sources less regular in $x$  but more regular in $t$ 
than the standard condition $f\in L^1_{loc}(\RR\,;\, L^2(\RR^d))$.

\section{Analysis of the two scale ansatz \eqref{eq:utaylor} }
\label{sec.twoscale}
Revisit the standard method of two scale asymptotic expansions for the wave equation 
\eqref{eq:homprob}. We depart from the textbooks \cite{BP}, \cite{BLP}, \cite{SP} in 
several ways. First, in those
 books the method is usually applied to an elliptic equation and the wave 
equation is only said to be treated similarly. Second, we do not content ourselves with 
computing the first two or three terms and giving a recurrence for the other ones.
Exact combinatorial formulas are given for terms of all orders  in the {\it ansatz} \eqref{eq:utaylor}.

Infinite order asymptotic expansions require that the source term $f$ be 
infinitely smooth. The periodic coefficients $\rho(y)$ and $a(y)$ are 
assumed only to be in 
$L^\infty(\TT^d)$. 
 
\subsection{Ansatz and first hierarchy}
Let $\caA_{yy}, \caA_{xy},\caA_{xx}$ be the second order partial differential operators 
defined in \eqref{eq:operators}.  
Consider the two scale power series \eqref{eq:utaylor}, 
and the corresponding formula
\eqref{eq:defW} for the right hand side.
All terms $u_n(t,x,y)$ and $w_n(t,x,y)$ are periodic in $y$,
equivalently defined for $y\in\TT^d$. 
The relation 
\eqref{eq:defW} 
 is equivalent to
\begin{equation}
\label{eq:mary}
\bigg[
	 \rho(y) \partial_t^2
  -\frac{1}
  {\eps^2}
  \mathcal{A}_{yy}
  \ -\ 
   \frac{1}
  {\eps}
  \mathcal{A}_{xy}
  \ -\ 
  \mathcal{A}_{xx}
  \bigg]
  \sum_{n=0}^\infty
  \eps^n\, u_n(t,x,y)
  \ =\
  \sum_{n=-2}^\infty
 \eps^n \,w_n(t,x,y)
\end{equation}
as formal  Laurent series in $\eps$.
 Equation  \eqref{eq:mary}  at order $\eps^n$ reads 
\begin{equation}
 \begin{aligned}
 \label{eq:hierarchy1}
&\eps^{-2}:&
-\caA_{yy} u_0 =w_{-2} \,,&
\cr
&\eps^{-1}:&
-(\caA_{yy} u_1+\caA_{xy}) u_0 =w_{-1}\,,&
\end{aligned}
\end{equation}
and, for $k\geq0$, the  coefficient of $\eps^k$ is
\begin{align}
 \label{eq:epsk}
\rho(y) \partial_t^2 u_k-
 ( \mathcal{A}_{yy} u_{k+2}
 \ +\ 
 \mathcal{A}_{xy} u_{k+1} 
 \ +\ 
 \mathcal{A}_{xx} u_k  )
  \ =\
  w_k\,.
\end{align}
In both classical and  criminal strategies we construct
profiles so that  the  $w_k$ do not depend 
on the fast variable $y$. The equation \eqref{eq:homprob} 
corresponds to $w_0=f$ and  $w_k=0$, $k\ne 0$.

\subsection{Projections and the hierarchy}
The analysis of \eqref{eq:hierarchy1}, \eqref{eq:epsk} pivots around the 
second order symmetric elliptic operator $\mathcal{A}_{yy}:H^1(\TT^d)\to H^{-1}(\TT^d)$.
Denote by $\pi$ the $L^2(\TT^d)$ orthogonal projection on constants,
$$
  \pi g \ :=\ \frac{1}{|\TT^d|} \int_{\TT^d} g(y)\ dy\,.
$$
This operator $\pi$ coincides with the action of $g$ as a distribution on the test
function $1$.  It is therefore a well defined operator on all 
periodic distributions. 
This operator extends to functions of $t,x,y$ by acting only on the last variable,
$$
(\pi g)(t,x) \ :=\ \frac{1}{|\TT^d|} \int_{\TT^d} g(t,x,y)\ dy\,.
$$

\begin{lemma}[\bf Cell Problem]
\label{lem:cell}
The operators in \eqref{eq:operators}  satisfy
$$
\pi \mathcal{A}_{yy}=0 \, , \quad \mathcal{A}_{yy}\pi =0
\quad \mbox{ and } \quad \pi\mathcal{A}_{xy}\pi=0 \, .
$$
The nullspace of $\mathcal{A}_{yy}$ is equal to the space of constant functions, 
i.e. $\pi H^1(\TT^d)$.  The image,  ${\rm Range} \,\mathcal{A}_{yy}$, is the 
subspace of mean zero functions, i.e. $(I-\pi) H^{-1}(\TT^d)$. 
Therefore $\mathcal{A}_{yy}$ is a bijection 
$
 (I-\pi) H^{1}(\TT^d) 
 \ \to\
(I-\pi) H^{-1}(\TT^d)  
$.   
$\mathcal{A}_{yy}^{-1}$ denotes its inverse.
\end{lemma}

\noindent
{\bf Proof.}   A classical
application of the Lax-Milgram Lemma.
\hfill
$\Box$
\vskip.3cm

 To solve 
\eqref{eq:hierarchy1} and \eqref{eq:epsk},
 these equations will be projected 
by $\pi$ (yielding the non-oscillatory hierarchy) and $(I-\pi)$ (leading to 
the oscillatory hierarchy) and solved separately.

\subsubsection{The oscillatory hierarchy}

Consider power series $U$ and $W$ for which $(I-\pi)w_n=0$ for all $n\geq -2$. 
Equations \eqref{eq:hierarchy1} and \eqref{eq:epsk} are multiplied on the left by $(I-\pi)$. 
Using $(I-\pi)w_{-2}=0$ the first line  of \eqref{eq:hierarchy1} becomes 
$$
0=(I-\pi)\mathcal{A}_{yy}u_0=(I-\pi)\mathcal{A}_{yy}(I-\pi)u_0.
$$
Lemma \ref{lem:cell} shows that this is equivalent to $(I-\pi)u_0=0$,
The oscillatory part of $u_0$ vanishes.

Since $\pi \caA_{xy} \pi =0$, one has $\pi \caA_{xy}u_0 = \pi \caA_{xy}\pi u_0=0$. 
Thus, the second line of 
\eqref{eq:hierarchy1}  shows that $(I-\pi)w_{-1}=0$  if and only if
\begin{equation}
\label{eq:oscillu1}
(I-\pi)u_1=  -\caA_{yy}^{-1}\,(I-\pi)\,\caA_{xy} \pi u_0 = 
-\caA_{yy}^{-1}\,\caA_{xy} \pi u_0
\ :=\ 
\chi_1(y, \partial_x) \pi u_0
\,.
\end{equation}
Next, derive analogous formulas expressing $(I-\pi)u_k$
in terms of the  $\pi u_j$ with $j<k$. 
Since by assumption $(I-\pi)w_k=0$ for all $k\geq 0$, \eqref{eq:epsk} leads to 
\begin{align*}
(I-\pi)
\Big[ \rho(y)\partial_t^2 u_k -
\mathcal{A}_{yy}u_{k+2} \ -\ 
 \mathcal{A}_{xy} u_{k+1} 
 \ -\ 
 \mathcal{A}_{xx} u_k 
  \Big]
  \ =\ 0\,.
\end{align*}
By Lemma \ref{lem:cell} this is equivalent to
\begin{align}
\label{eq:epsorthog}
(I-\pi)u_{k+2}
\ =\ 
-\,\mathcal{A}_{yy}^{-1}
\,
\big(
I-\pi
\big)
\Big[  
 \mathcal{A}_{xy} u_{k+1} 
 \ +\ 
 (\mathcal{A}_{xx}   
-\rho(y)\partial_t^2) u_k
  \Big]
  \,.
\end{align}
Equation \eqref{eq:epsorthog} expresses the oscillatory
part of $u_{k+2}$ in terms of earlier profiles. It can be further simplified 
by rewriting the earlier profiles as $u_j=\pi u_j + (1-\pi)u_j$, the sum of 
non oscillatory and oscillatory parts. Then express the $(1-\pi)u_j$ parts in terms
of still earlier profiles, and so on.  In this way
the oscillatory parts can be eliminated yielding a relation
determining the oscillatory parts in terms of the nonoscillatory parts
that is made explicit  
 in Theorem \ref{thm:osc}. In \eqref{eq:oscillu1} 
an operator $\chi_1$ was introduced.  This definition is now
extended to higher order. 

\begin{definition}
\label{def.chi-k}
Set $\chi_{-1} :=  0$, $\chi_0  :=  I$. 
For $k\ge 1$ define operators mapping functions of 
$t,x$ to functions of $t,x,y$   by
\begin{equation}
\begin{aligned}
\label{eq:recur1}
 \chi_k(y,\partial_t,\partial_x)  \ :=\  - \caA_{yy}^{-1} (I-\pi) 
\Big[ \caA_{x y} \chi_{k-1}\ +\  (\caA_{xx}-\rho(y)\partial_t^2) \chi_{k-2} \Big] 
\ =\ (I-\pi)\chi_k
.
\end{aligned}
\end{equation}
\end{definition}

This recovers the previous
definition of $\chi_1= - \caA_{yy}^{-1} \caA_{x y}= - \caA_{yy}^{-1} (I-\pi)\caA_{x y}$, where the last equality follows from Lemma \ref{lem:cell}. 
The operators $\chi_k$ depend on $y$. 
The $y$-dependence arises only from the coefficients $a(y), \rho(y)$.
To show that the above definition makes sense, it suffices to prove that for any 
smooth function $\varphi\in C^\infty(\mathbb{R}^{1+d})$ and every $(t,x)\in\mathbb{R}^{1+d}$
the argument of $\caA_{yy}^{-1}$, namely
\begin{equation*}
(I-\pi) 
\Big[ \caA_{x y} \chi_{k-1}\ +\  (\caA_{xx}-\rho\partial_t^2) \chi_{k-2} \Big]\varphi(t,x)\, ,
\end{equation*}  
belongs to the range of $\caA_{yy}$. This is verified in the next Lemma.

\begin{lemma}
\label{lem:wellposedchi}
For all $k\geq 1$ the following holds. For every function 
$\varphi\in C^\infty(\mathbb{R}^{1+d})$ and every $(t,x)\in\mathbb{R}^{1+d}$ one has that
$$
\Big[ \caA_{x y} \chi_{k-1}\ +\  (\caA_{xx}-\rho\partial_t^2) \chi_{k-2} \Big]\varphi(t,x) 
$$
belongs to $H^{-1}(\TT^d)$. In particular $\chi_k\varphi(t,x)\in (I-\pi)H^1(\TT^d)$. 
Furthermore, for any $k\geq1$, there exist coefficients $c_{\beta,k}\in (I-\pi) H^1(\TT^d)$ 
 such that
\begin{equation}
\label{eq:chikcalpha}
\chi_k(y,\partial_t,\partial_x) \ =\ 
\sum_{|\beta|=k}
c_{\beta,k}(y)\,\partial_{t,x}^\beta \, .
\end{equation}
In particular, $\chi_k$ is a homogeneous operator of degree $k$ in
$\partial_{t,x}$. 
\end{lemma}

{\bf Proof.} 
The proof is  by induction on $k$. 
For $k=1$, $\chi_1$ is a first-order operator in $x$ with $(I-\pi) H^1(\TT^d)$ 
coefficients, since $\caA_{x y}\varphi$ belongs to $H^{-1}(\TT^d)$. 

Assume the statement for $k\geq1$ and prove it for $k+1$. 
For a function $\varphi\in C^\infty(\mathbb{R}^{1+d})$ compute
\begin{equation}
\begin{aligned}
\label{eq:wellposedchi}
\big[
\caA_{xy}\chi_k + (\caA_{xx}  -  \rho & (y)\partial_t^2) \chi_{k-1}
\big]
\varphi(t,x)
\cr
\ =\ 
 \dive_x  \big(a  & (y)\grad_y\chi_k \varphi(t,x)\big) +
  \dive_y \big(a(y)\grad_x\chi_k\varphi(t,x)\big)
\cr
&\ +\
\dive_x \big( a(y)\grad_x\chi_{k-1}\varphi(t,x)\big)-\rho(y)\partial_t^2\big(\chi_{k-1}
\varphi(t,x)\big)\,.
\end{aligned}
\end{equation}  
By the induction hypothesis $\chi_k\varphi(t, x)$ and $\chi_{k-1}\varphi(t,x)$ are 
in $H^1(\TT^d)$.  Therefore,
all terms on the right hand side of \eqref{eq:wellposedchi} are 
 in $H^{-1}(\TT^d)$. 
In particular $\chi_{k+1}\varphi(t,x)\in H^1(\TT^d)$, which is the claimed result. 

Since the operator $\caA_{x y}$ is homogeneous of degree one and $(\caA_{xx}-\rho\partial_t^2)$ is homogeneous
of degree two,  it follows that $\chi_{k+1}$ is homogeneous of degree $k+1$. 
\hfill
\qed

\begin{remark} 
\label{rem:rhoconstant1}
{\bf i.}   If $\rho$ is independent of $y$, the fact that 
$(I-\pi)\chi_0=0$ implies that 
$\chi_2(y,\partial_t,\partial_x)$ does not depend on $\partial_t$. 
{\bf ii.}  In this case an induction 
on $k$ shows that  
for any $k\geq1$,
$\chi_k(y,\partial_t,\partial_x)$
contains only time derivatives of order   $\le k-2$. 
\end{remark}

The first structural result concerns  formal power series $U$
for which the oscillatory parts $(I-\pi)w_n$ vanish.

\begin{theorem}  
\label{thm:osc}
Fix $k\in\mathbb Z$ with $k\geq -2$. 
For a formal power series $U$ and corresponding $W$ the following are
equivalent.

{\bf i.}  
For $-2 \le j\le k$ one has
\begin{equation}
(I-\pi)w_j =0
\,.
\end{equation}

{\bf ii.}  For $0 \le \ell\le k+2$ one has
\begin{equation}
\label{eq:osc}
 (I-\pi) u_{\ell}  \ = \ \sum_{n=1}^{\ell} \chi_n   \pi u_{\ell-n} \,.
\end{equation}
\end{theorem}

{\bf Proof.}  
For $k=-2$ the statement follows directly by recalling that $(I-\pi)w_{-2}=0$ if and only if $(I-\pi)u_0=0$. For $k=-1$ one has
\begin{equation}
\label{eq:IminusPiw-1}
(I-\pi)w_{-1}=-(I-\pi)\Big(\caA_{yy}u_1 + \caA_{x y} u_0\Big)\,.
\end{equation}
Lemma \ref{lem:cell}  implies  that
\begin{equation}
\label{eq:AyyIminusPi} 
\pi\caA_{yy}=0,\quad\quad
{\rm and},
\qquad(1-\pi)\caA_{yy}u_{k} = \caA_{yy}(1-\pi)u_{k}\quad\text{for all }k\geq 0.
\end{equation} 
Using \eqref{eq:AyyIminusPi} along with $u_0=\pi u_0$ and multiplying \eqref{eq:IminusPiw-1} by $-\mathcal{A}_{yy}^{-1}$  yields 
\begin{equation*}
-\caA_{yy}^{-1}(I-\pi)w_{-1}
\ =\
(I-\pi)u_1 + \caA_{yy}^{-1}(I-\pi)\caA_{x y} \pi u_0
\  =\
(I-\pi)u_1 -\chi_1\pi u_0\, .  
\end{equation*}
Since $(I-\pi)w_{-1}=0$ is equivalent to $-\caA_{yy}^{-1}(I-\pi)w_{-1}=0$, this 
proves the case 
$k=-1$ of the Theorem. 

For $k\geq 0$  reason by induction. 
Assume the case $k-1$ and prove the case $k$. The induction hypothesis is
\begin{equation}
\label{eq:recur4b}
(I-\pi) u_{\ell}  \ = \ \sum_{n=1}^{\ell} \chi_n   \pi u_{\ell-n}\quad\text{for } 0\le \ell\le k+1
\,,
\end{equation}
if and only if $(I-\pi)w_j=0$ for $-2\le j\le k-1$. For the inductive step 
 need  to treat  $j=k$ and $l=k+2$. For $k\geq 0$ one has 
\begin{equation}
\begin{aligned}
\label{eq:recur3}
(I-\pi)w_k\ =\ -(I-\pi)\Big( \caA_{yy}  u_{k+2} \ + \ \caA_{x y}  u_{k+1} 
\ + \ (\caA_{xx} - \rho(y)\partial_t^2)  u_k \Big) \,.
\end{aligned}
\end{equation}
Exploiting \eqref{eq:AyyIminusPi} and multiplying by $-\mathcal{A}_{yy}^{-1}$ yields  
\begin{equation}
\label{eq:recur4}
-\caA_{yy}^{-1}(I-\pi)w_{k}\ =\
(I-\pi)  u_{k+2} \ + \  \caA_{yy}^{-1} (I-\pi) 
\Big( \caA_{x y} u_{k+1} \ +\ (\caA_{xx} - \rho(y)\partial_t^2)  u_{k} \Big) \, .
\end{equation}
Expressing each profile in \eqref{eq:recur4} as a sum of its oscillatory and non-oscillatory part and recalling the definition
$\chi_1=-\caA_{yy}^{-1}(I-\pi )\caA_{xy}$,
yields
\eqref{eq:recur4} as
\begin{align}
\label{eq:walter}
-\caA_{yy}^{-1}(I-\pi)w_{k}& =
(I-\pi)  u_{k+2} - \chi_1\pi u_{k+1} + \caA_{yy}^{-1} (I-\pi)\Big((\caA_{xx} - \rho(y)\partial_t^2) \pi u_{k}\Big)
\cr
 &+\caA_{yy}^{-1} (I-\pi)\Big( \caA_{x y}(I-\pi) u_{k+1} \ +\ (\caA_{xx} - \rho(y)\partial_t^2) (I-\pi) u_{k} \Big)\,.
\end{align}
Use that $(I-\pi)w_{k}=0$ if and only if $\caA_{yy}^{-1}(I-\pi)w_{k}=0$. 
Thus $(I-\pi)w_k=0$ if and only if
the right hand side of \eqref{eq:walter} vanishes.  Using the induction hypothesis, \eqref{eq:recur4b} holds for 
$(I-\pi) u_{k+1}$ and $(I-\pi) u_{k}$.   This yields 
\begin{align*}
(I-\pi) u_{k+2}
\ &=\
\chi_1\pi u_{k+1} - \caA_{yy}^{-1} (I-\pi)\Big((\caA_{xx} - \rho(y)\partial_t^2) \pi u_{k}\Big)\\
&\quad -\caA_{yy}^{-1} (I   -\pi)\left( \caA_{x y}\sum_{n=1}^{k+1} \chi_n   \pi u_{k+1-n} \ +
 (\caA_{xx} - \rho(y)\partial_t^2) \sum_{n=1}^{k} \chi_n   \pi u_{k-n} \right)\,.\\
\ &=\ \chi_1\pi u_{k+1} - \caA_{yy}^{-1}(I-\pi)\Big((\caA_{xx} - \rho(y)\partial_t^2) \pi u_{k}+\caA_{x y}\chi_1\pi u_k\Big) \\
&\quad -\caA_{yy}^{-1} (I-\pi)\left(
\caA_{x y}\sum_{n=2}^{k+1} \chi_n   \pi u_{k+1-n} + (\caA_{xx} - 
\rho(y)\partial_t^2)\sum_{n=1}^{k} \chi_n   \pi u_{k-n}\right).
\end{align*}
By definition $- \caA_{yy}^{-1} (I-\pi)\Big((\caA_{xx} - \rho(y)\partial_t^2) \pi u_{k}+\caA_{x y}\chi_1\pi u_k\Big)=\chi_2\pi u_k$.
Therefore
\begin{align*}
(I-\pi) &u_{k+2}\\
=&\chi_1\pi u_{k+1} + \chi_2 \pi u_k - \caA_{yy}^{-1} (I-\pi)\Big(
\caA_{x y}\sum_{n=2}^{k+1} \chi_n   \pi u_{k+1-n} + (\caA_{xx} - 
\rho(y)\partial_t^2)\sum_{n=1}^{k} \chi_n   \pi u_{k-n}\Big)\\
=& \chi_1\pi u_{k+1} + \chi_2 \pi u_k + \sum_{n=1}^k
\Big(-\caA_{yy}^{-1} (I-\pi)\Big)\Big[\caA_{x y}\chi_{n+1} + (\caA_{xx} - 
\rho(y)\partial_t^2)\chi_n\Big]\pi u_{k-n}\\
=&\chi_1\pi u_{k+1} + \chi_2 \pi u_k + \sum_{n=1}^k\chi_{n+2}\pi u_{k-n}
=\sum_{n=1}^{k+2}\chi_n\pi u_{k+2-n}\, ,
\end{align*}
where  the last line  uses the definition (\ref{eq:recur1}) 
of $\chi_{n+2}$. The last identity is the  desired formula for $(I-\pi) u_{k+2}$.
The proof is complete.
\hfill
\qed
\vskip.2cm

\begin{remark}
\label{rem:BP1}
Theorem \ref{thm:osc} has a particularly elegant form
for profiles so that $(1-\pi)w_\ell=0$ for all $\ell$.  This holds if 
and only if 
the  formal power series in $\eps$ for $u$ is given in terms of the 
series 
for $\pi u$ by 
$$
\sum_{n=0}^\infty \eps^n\, u_n\ =\ 
\Big(\sum_{\ell=0}^\infty \eps^\ell \chi_\ell\Big)
\,   
\Big( \sum_{k=0}^\infty\eps^k \pi u_k\Big)\,.
$$
The elliptic analogue was observed
by Bakhvalov and Panasenko \cite{BP}.
\end{remark}

\subsubsection{The nonoscillatory hierarchy} 
\label{sec:nonoscillatory}
Next analyse the equations determining the 
non oscillatory parts $\pi u_n$ of the profiles.
Equations \eqref{eq:hierarchy1} and \eqref{eq:epsk} are multiplied on the left by $\pi$. 
Since $\pi \mathcal{A}_{yy}=0$ and $\pi\mathcal{A}_{xy}\pi=0$, one has $\pi w_{-2}=0$ and $\pi w_{-1}=0$. 
For $k\geq 0$ the $\mathcal{A}_{yy}$ terms are eliminated and \eqref{eq:epsk} with $k\geq 0$ simplifies to
\begin{equation}
\label{eq:piwk1}
\pi w_k=
\pi
\Big[\rho(y)\partial_t^2 u_k  
\ -\ 
 \mathcal{A}_{xx} u_k - 
 \mathcal{A}_{xy} u_{k+1} 
  \Big].
\end{equation}

Exploiting \eqref{eq:piwk1} with $k=0$, 
writing $u_1=\pi u_1 + (1-\pi)u_1$ and using the recurrence 
from \eqref{eq:osc} yields 
\begin{equation*}
\pi \caA_{xy}u_1\ 
\stackrel{\pi\mathcal{A}_{xy}\pi=0}{=}
\
\pi \caA_{xy}(I-\pi)u_1
\stackrel{\eqref{eq:osc}}{=}
\pi \caA_{xy}\chi_1\pi u_0
\ \stackrel{\eqref{eq:recur1}}{=}\ 
-\,\pi 
\caA_{xy}\,
\caA_{yy}^{-1}\, (I-\pi)\,\caA_{xy} \pi u_0\, . 
\end{equation*}
Since  $u_0=\pi u_0$,  this yields
\begin{equation*}
\pi w_0=a_2^*(\partial_t,\partial_x) \pi u_0
\end{equation*}
with the homogenized wave operator defined as
\begin{equation}
\label{eq:Tom}
a_2^*(\partial_t,\partial_x) \ :=\ (\pi \rho)\partial_t^2
\ -\ 
\dive_x\,( \pi a ) \, \grad_x  \ +\ 
\pi \,
\caA_{xy}(I-\pi)
\caA_{yy}^{-1}(I-\pi)\,\caA_{xy}\,\pi\,.
\end{equation}

\begin{remark}
The homogenized wave operator $a^*_2$ coincides with the
 formula from classical homogenization theory 
\cite{BP}, \cite{BLP}, \cite{JKO}, \cite{SP}.
\end{remark}


\begin{definition} 
Scalar partial differential operators
$a^*_n(\partial_t,\partial_x)$ mapping functions of $t,x$ to functions of
$t,x$
 are defined for $n\geq1$ by
\begin{equation}
\label{eq:recur33}
a^*_n(\partial_t,\partial_x) \ = \ \pi \Big( 
\big(\rho(y)\partial_t^2-\caA_{xx}\big)\chi_{n-2}-
\caA_{x y}\chi_{n-1}\Big) . 
\end{equation}
\end{definition}

\begin{remark}
{\bf i.}  The operators $a_n^*$ have constant coefficients.  
{\bf ii.}  The operator 
 $a^*_n$ is homogeneous of degree $n$.
{\bf iii.}  The symbol  $a_n^*(\partial_t,\partial_x)$  
contains only even powers of $\partial_t$. 
{\bf iv.}  The definitions of $\chi_0,\chi_{-1}$ imply that $a_1^*=0$.
\end{remark}

\begin{theorem}
\label{thm:recur1} 
Suppose that the formal power series $U$ and corresponding $W$ 
satisfy the conditions of Theorem \ref{thm:osc} for some $k\in\mathbb{Z}$ with $k\geq -2$. 
Then $\pi w_{-2}=\pi w_{-1}=0$ and  
for $0\le j\le k+1$,
\begin{equation}
\begin{aligned}
\label{eq:recur2}
\pi w_j \ =\ 
\sum_{n=0}^{j}\,
a^*_{n+2}(\partial_t,\partial_x) \pi  u_{j-n}  
\,.
\end{aligned}
\end{equation}
\end{theorem}

\begin{remark}
\label{rem:BP2}
The result is particularly elegant for profiles so that $(1-\pi )w_n=0$ for all 
$n$.  In that case the formal power series in $\eps$ for the residual
is given in terms of the nonosicllating parts by
$$
\sum_{j=0}^\infty \eps^{j} \pi w_j 
 \ =\ 
\Big( \sum_{n=0}^\infty \eps^n a_{n+2}^*\Big)\Big(\sum_{m=0}^\infty\eps^m \pi u_m\Big)\,.
$$
The elliptic analogue was observed in \cite{BP}.
\end{remark}

{\bf Proof.} The cases $k=-2$ and $k=-1$ have already been discussed above. Let $k\geq 0$ and fix 
$0\leq j\leq k+1$. 
Using 
 $\pi \caA_{xy} \pi = 0$ and $\pi \caA_{yy}=0$ provides
\begin{equation}
\label{eq:recur6}
\pi w_{j} \ =\ 
\pi \Big((\rho(y)\partial_t^2-\caA_{xx})\pi  u_{j} + (\rho(y)\partial_t^2-\caA_{xx})(I-\pi) u_{j}
 -\caA_{xy} (I-\pi)  u_{j+1} 
\Big)  \, .
\end{equation}
Since we assumed that the conditions of Theorem \ref{thm:osc} hold for $k$ and since $j, j+1\leq k+2$, 
we can replace $(I-\pi)u_l$ in (\ref{eq:recur6}), for $l=j, j+1$ according to 
formula \eqref{eq:osc}. This yields
\begin{align*}
\pi w_j \ &=\ 
\pi (\rho(y)\partial_t^2-\caA_{xx}) \pi  u_{j}+
\pi\Big((\rho(y)\partial_t^2-\caA_{xx})\sum_{n=1}^j\chi_n\pi u_{j-n}
-\caA_{xy}\sum_{n=1}^{j+1}\chi_n\pi u_{j+1-n}\Big) \\
\ &=\ \pi\Big((\rho(y)\partial_t^2-\caA_{xx}) -\caA_{xy}\chi_1\Big)\pi u_j 
+ \pi \Big(\sum_{n=1}^j(\rho(y)\partial_t^2-\caA_{xx})\chi_n\pi u_{j-n}-
\sum_{n=2}^{j+1}\caA_{xy}\chi_n \pi u_{j+1-n}\Big)
\end{align*}
Regrouping terms and recalling that $\chi_0=I$ yields
\begin{equation}
\label{eq:piwk}
\pi w_j\ =\ 
\pi \sum_{n=0}^{j} \Big((\rho(y)\partial_t^2-\caA_{xx})\chi_{n}-\caA_{xy} \chi_{n+1}\Big) 
\pi u_{j-n}   \, .
\end{equation}
By definition of the effective operators $a^*_{n+2}$, Equation \eqref{eq:piwk} is equivalent to
\begin{equation}
\label{eq:alla}
\pi w_j\ =\ 
\sum_{n=0}^{j} a^*_{n+2}(\partial_t,\partial_x) \pi  u_{j-n}
 \, .
\end{equation}
This completes the proof. 
\hfill
\qed
\vskip.2cm

\begin{remark}  $a^*_n$ is 
a homogeneous polynomial of degree $n$ in $(\partial_t,\partial_x)$
Formula (\ref{eq:recur33}) shows that  the highest degree 
of $\partial_t$ in $a^*_n$ comes from $\chi_{n-1}$ or $\partial^2_t\chi_{n-2}$. 
When $\rho$ is independent of $y$, Remark \ref{rem:rhoconstant1} yields that $\chi_{n-1}$ and $\chi_{n-2}$ are of 
degree $\leq n-3$ and $\leq n-4$, respectively, with respect to time $t$. Therefore, when
$\rho$ is constant,
$a^*_n$ is of  order $\le n-2$ in $\partial_t$ for $n>2$.
\end{remark}

The next result shows that the equation \eqref{eq:alla} has half
as many terms as it seems. The proof depends on   a precise
combinatorial formula for
$\chi_k$. The elliptic analogue of Theorem \ref{thm:recur2}
was proved  
by a quite different variational argument in \cite{SmCh}. 

\begin{theorem}
\label{thm:recur2}
For any odd $n\ge 1$, the homogenized operator $a_n^*$  vanishes. 
That is  for $m\ge 0$, $a^*_{2m+1}=0$.
\end{theorem}

{\bf Proof.}
Introduce
$$
C_1\ :=\  - \caA_{yy}^{-1} (I-\pi) \caA_{x y} \quad \mbox{ and } \quad 
C_2\ 
:=\  -\caA_{yy}^{-1} (I-\pi) (\caA_{xx}-\rho(y)\partial_t^2) \,.
$$
The operator $\caA_{yy}^{-1}(I-\pi)$  acts only on the $y$ variable and 
is continuous from $H^{-1}(\TT^d)\to H^{1}(\TT^d)$.
The operators $C_j$   are homogeneous polynomials
of degree $j$ in $(\partial_t,\partial_x)$, whose coefficients are operators in $y$. 
In the proof we integrate by parts with respect to $y$ and not with respect to
$t,x$.  With these $C_j$,
\ref{eq:recur1} yields 
\begin{equation}
\label{eq:albert}
\chi_k\ =\  C_1 \chi_{k-1} + C_2 \chi_{k-2}\,,
\qquad
 k\geq1 
 \,.
\end{equation}
Replace $\chi_{k-1}$ and $\chi_{k-2}$
using the two earlier instances of the recurrence. 
Continuing, leads to an expression
$$
\chi_k\ =\ W_k \chi_0,
$$
where only the earliest operator $\chi_0$ appears.
Equation \eqref{eq:albert} implies that 
\begin{equation}
\label{eq:recur7}
\begin{aligned}
W_k & \ =\  
C_1 W_{k-1} + C_2 W_{k-2}\ 
\,.
\end{aligned}
\end{equation}
{\sl Equation \eqref{eq:recur7} implies that $W_k$
 is the sum 
of all words written with the two "letters" $C_1$ and $C_2$ 
such that the number of letters satisfies $\# C_1 + 2\,\# C_2 = k$. }
Each word is a homogeneous differential operator of degree $k$ in $\partial_{x}$. 
Separating the words into two groups, those that end in $C_1$
and those that end in $C_2$ implies that 
\begin{equation*}
W_k
\ =\
W_{k-1} C_1 + W_{k-2} C_2\,. 
\end{equation*}

Denote with an exponent $T$ the $L^2(\TT^d)$ adjoint.  
Integration by parts in $y$ shows that $\caA_{xy}^T=-\caA_{x y}$, while $\caA_{yy}$ and $\caA_{xx}$ are selfadjoint. 
Define  operators
\begin{align*}
D_1& := - \caA_{x y} \caA_{yy}^{-1} (I-\pi) = - C_1^T\,,
\cr
D_2 & := -(\caA_{xx}-\rho(y)\partial_t^2) \caA_{yy}^{-1} (I-\pi) =  C_2^T \, .
\end{align*}
An induction shows that
$$
W_k^T = (-1)^k Z_k \quad \mbox{ with } \quad 
Z_k = D_1 Z_{k-1} + D_2 Z_{k-2} \,.
$$
{\sl
Therefore $Z_k$ is the sum of all words written 
with the two letters $D_1$ and $D_2$ 
such that the number of letters satisfy $\# D_1 + 2\,\#D_2 = k$. }

Introduce $G := \caA_{yy}^{-1} (I-\pi)$ that satisfies
$$
C_1G = G D_1, \qquad C_2 G = G D_2 , \qquad 
W_k G = G Z_k \, . 
$$
Since $\chi_0(y) =I$, definition (\ref{eq:recur33}) can be rewritten, 
by using $\caA_{ x y}^T=-\caA_{xy}$, $(\caA_{xx}-\rho\partial_t^2)^T = (\caA_{xx}-\rho\partial_t^2)$ 
and $W_k^T = (-1)^k Z_k$, as 
$$
\begin{aligned}
a^*_k(\partial_t,\partial_x) & = \int_{\T^d} \Big( (\rho\partial^2_t-\caA_{xx})W_{k-2} \chi_0(y)
-\caA_{x y} W_{k-1} \chi_0(y) \Big)\chi_0(y)\, dy 
\cr
& = \int_{\T^d} \Big(W_{k-2} \chi_0(y)(\rho\partial^2_t-\caA_{xx})\chi_0(y)
+ W_{k-1} \chi_0(y) \caA_{xy} \chi_0(y)\Big)  dy 
\cr
& = (-1)^k \int_{\T^d} \chi_0(y) \Big(
Z_{k-2}(\rho\partial^2_t-\caA_{xx})\chi_0(y)
- Z_{k-1} \caA_{xy} \chi_0(y) \Big) dy \, .
\end{aligned}
$$
The properties of $Z_k$ and $W_k$ imply that 
$$
\begin{aligned}
-Z_{k-1} \caA_{x y} & =\  \big(D_1 Z_{k-2} + D_2 Z_{k-3}\big) \caA_{x y} 
\cr
& =\big( (\rho(y)\partial_t^2-\caA_{xx}) G Z_{k-3}-\caA_{xy} G Z_{k-2}\big)\caA_{x y} 
\cr
& =\big((\rho(y)\partial_t^2-\caA_{xx}) W_{k-3}G-\caA_{xy} W_{k-2}G \big)\caA_{x y} 
\cr
& = \
\big(\rho(y)\partial_t^2-\caA_{xx}\big) W_{k-3} C_1-\caA_{xy} W_{k-2} C_1.
\end{aligned}
$$
Similarly, 
$$
\begin{aligned}
Z_{k-2}\big(\rho(y)\partial^2_t-\caA_{xx}\big) & =\
\big(D_1 Z_{k-3} + D_2 Z_{k-4}\big)\big(\rho(y)\partial^2_t-\caA_{xx}\big)
\cr
& = 
((\rho(y)\partial^2_t-\caA_{xx})G Z_{k-4}-\caA_{xy} G Z_{k-3}) (\rho(y)\partial^2_t-\caA_{xx})
\cr
& = ((\rho(y)\partial^2_t-\caA_{xx})W_{k-4}G -\caA_{x y} W_{k-3} G )
 (\rho(y)\partial^2_t-\caA_{xx})
\cr
& = \
(\rho(y)\partial^2_t-\caA_{xx})W_{k-4}C_2-\caA_{x y} W_{k-3} C_2\,.
\end{aligned}
$$
Summing yields 
\begin{align*}
Z_{k-2}(\rho(y)\partial^2_t-&\caA_{xx}) -Z_{k-1} \caA_{x y}=(\rho(y)\partial^2_t-\caA_{xx})(W_{k-3} C_1+W_{k-4}C_2)\\
&-
\caA_{x y}( W_{k-2} C_1 + W_{k-3} C_2)
 + (\rho(y)\partial^2_t-\caA_{xx})W_{k-2}-\caA_{x y}W_{k-1}.
\end{align*}
Therefore   
\begin{align*}
a^*_k(\partial_t,\partial_x) &= (-1)^k \int_{\T^d} \chi_0(y) \Big(
(\rho\partial_t^2-\caA_{xx}) W_{k-2} \chi_0(y) 
 -\caA_{x y} W_{k-1} \chi_0(y) 
\Big) dy
\cr &= (-1)^k a^*_k(\partial_t,\partial_x) \, .
\end{align*}
For odd $k$, this implies $a_k^*(\partial_t,\partial_x)=0$.
\hfill
\qed
\vskip.2cm
Consider the homogenization problem \eqref{eq:homprob}. 
The goal  is to describe the 
behavior of the solution $u^\eps$ by investigating formal power series $U$ as in \eqref{eq:utaylor}. 
The classical algorithm is to choose the series $U$ such that $W-f\sim 0$, i.e to choose the profiles $u_n$ such that for all $t,x,y$
$$
w_0(t,x,y)=f(t,x)\,,
\qquad
\forall \, {0\neq n \ge -2},
\quad
w_n=0\,.
$$
The equation $W-f\sim 0$ is satisfied if and only if $\pi(W-f)\sim 0$ and $(I-\pi)(W-f)\sim 0$. 
The source term $f(t,x)$ is smooth and non oscillatory, $(I-\pi)f=0$. For the power series $U$ 
it follows that
 $(I-\pi)w_k=0$ for all $k\geq -2$. According to Theorem \ref{thm:osc} such a power series satisfies 
\begin{equation}
\label{eq:i-piuell1}
(I-\pi) u_{\ell}  \ = \ \sum_{n=1}^{\ell} \chi_n   \pi u_{\ell-n}\quad\text{for all }l\geq 0. 
\end{equation} 
Next analyse the equations determining the non oscillatory parts $\pi u_n$ of the profiles.  
Theorem \ref{thm:recur1} shows that if the oscillatory parts satisfy
\eqref{eq:i-piuell1}, then
\begin{equation*}
\pi w_j \ =\ 
\sum_{n=0}^{j}\,
a^*_{n+2}(\partial_t,\partial_x) \pi  u_{j-n}\,.
\end{equation*}
Theorem \ref{thm:recur2} implies that only terms of the same parity appear,  
$$
\pi
w_j
\ =\
a^*_2(\partial_t,\partial_x) \pi u_j
\ +\
\sum_{2\le n\atop 2n +k=j+2}
 a_{2n}^*(\partial_t,\partial_x) \pi u_k
\,.
$$
The classical algorithm is to set $\pi w_0=f$ and $\pi w_j=0$ for $-2\leq j\neq 0$, which yields
the following hierarchy of equations for the 
$\pi u_k$,
\begin{equation}
\label{eq:hierarchy2}
\begin{aligned}
\eps^0: \qquad a^*_2(\partial_t,\partial_x) \pi u_0  \ &=\ f
\cr
\eps^1: \qquad a^*_2(\partial_t,\partial_x) \pi u_1  \ &=\ 0
\cr
\eps^2: \qquad a^*_2(\partial_t,\partial_x) \pi u_2  \ &=\  -a_4^*(\partial_t, \partial_x)\pi u_0
\cr
\eps^3: \qquad a^*_2(\partial_t,\partial_x) \pi u_3  \ &=\  - a_4^*(\partial_t, \partial_x)\pi u_1
\cr
\eps^4:\qquad a^*_2(\partial_t,\partial_x) \pi u_4  \ &= \ - a_4^*(\partial_t, \partial_x)\pi u_2 \ -\  a_6^*(\partial_t,\partial_x) \pi u_0
\cr
\eps^5:\qquad a^*_2(\partial_t,\partial_x) \pi u_5  \ &= \  -a_4^*(\partial_t, \partial_x)\pi u_3\ -\ a_6^*(\partial_t,\partial_x) \pi u_1
\cr
\eps^6: \qquad a^*_2(\partial_t,\partial_x) \pi u_6  \ &= \ -a_4^*(\partial_t, \partial_x)\pi u_4\ -\ a_6^*(\partial_t,\partial_x) \pi u_2
\ -\  a_8^*(\partial_t,\partial_x) \pi u_0
\cr
\eps^7: \qquad a^*_2(\partial_t,\partial_x) \pi u_7  \ &= \  -a_4^*(\partial_t, \partial_x)\pi u_5\ -\  a_6^*(\partial_t,\partial_x) \pi u_3
\ -\  a_8^*(\partial_t,\partial_x) \pi u_1
\end{aligned}
\end{equation} 
The $\eps^0$-order equation yields the classical homogenized wave equation 
$
a_2^*(\partial_t,\partial_x) \pi u_0 \ =\ f.
$

 \vskip.2cm
 
\subsection{Leap frog and secular growth}
\label{sec:secular}
The equations for the odd subscripts are decoupled from those
with even subscripts. The equations repeat in pairs. This is the
{\it leap frog structure of the non oscillatory hierarchy}.
Starting with $n=1$ one concludes by
induction in steps of two,  that $\pi u_n=0$ for all odd $n$.

The leap frog structure implies that secular growth
is slow. Without the leap frog structure one would have 
\hskip-2pt
\footnote{The notation $A\lesssim B$ means that there is a constant
$C>0$, independent of $A$ and $B$, so that $A\le C\, B$. }
$\,|u_n| \lesssim t^n$ 
instead of the  $t^{n/2}$ in next theorem.

\begin{theorem}[\bf Secular growth]
\label{thm:secular}
If there is a $\ut>0$ so that $f=0$ for $t>\ut$, 
then for each non zero $\alpha\in \NN^{1+d}\setminus \{0\}$ and every $k=0,1,2,\dots$ there exists a constant
$C$ depending on $f,\alpha$ and $k$ so that for all $t\ge 0$,
\begin{equation}
\label{eq:leapfrog}
\big\|
\partial_{t,x}^\alpha u_{2k}(t)
\,,\,
 \partial_{t,x}^\alpha u_{2k+1}(t)\big\|_{L^2(\RR^d\times\TT^d) }
\ \le \ C \, \langle t \rangle^k\,,
\qquad
\langle t \rangle := \sqrt{1+t^2}\,.
\end{equation}
\end{theorem}

\begin{remark}
\label{rem:Mark}
Estimate \eqref{eq:leapfrog} provides a bound on the derivatives of the $u_n$ but
not on the $u_n$
themselves.  To estimate
$u_{2k}$ or $u_{2k+1}$ use $u=\int_0^t \partial_tu\,dt$ to find
\begin{equation}
\label{eq:leapfrog2}
\big\|
u_{2k}(t)
\,,\,
u_{2k+1}(t)\big\|_{L^2(\RR^d\times\TT^d) }
\ = \ 
O(\langle t \rangle^{k+1})\,.
\end{equation}
\end{remark}

\noindent
{\bf Proof of Theorem.}
The leading term $u_0(t,x)$ satisfies
$$
a^*_2(\partial_t,\partial_x) \pi u_0\ =  \ f\,,
\qquad
\pi u_0\ =\ 0 \quad {\rm for}\quad
t<0\,.
$$
Since $f\in C^\infty_0(\RR\,;\, H^s(\RR^d))$ for all $s$, it follows
that  for $0\ne \alpha\in \NN^{1+d}$,
$\partial_{t,x}^\alpha \pi u_0\in L^\infty(\RR\,;\, L^2( \RR^d\times\TT^d
))$.
Since $u_0=\pi u_0$ this finishes the analysis of $u_0$.

One has $\pi u_1=0$.
Equation \eqref{eq:oscillu1} implies that the oscillatory part of $u_1$
satisfies
$$
\partial_{t,x}^\alpha (I-\pi )u_1
\ =\
 -\partial_{t,x}^\alpha \mathcal{A}_{yy}^{-1}\, \mathcal{A}_{xy} \, u_0
\ \in \
L^\infty(\RR\,;\, L^2(\RR^d\times\TT^d))\,.
$$
This completes the analysis of $u_1$ and therefore the case $k=0$ of the
Theorem.

The proof is by induction on $k$.
 Assuming the
result for indices $ \le k$ it
 suffices to prove the case $k+1$.

First estimate the $\pi $ projections.
Since $2(k+1)+1$ is odd,  $\pi u_{2(k+1)+1}=0$.
To estimate $\pi u_{2k+2}$,
use the equation
$$
a^*_2(\partial_t,\partial_x) \pi u_{2k+2}
\ =\
\ -\
\sum_{2\leq n \atop n +j=k+2}
 a_{2n}^*(\partial_t,\partial_x) \pi u_{2j}\,.
$$
The case
$k$  of \eqref{eq:leapfrog}  bounds the right hand side.
Since $a_{2n}^*$ is a sum of derivatives,
the inductive hypothesis
implies that for all $\beta$ including $\beta=0$, 
$$
\big\|
\partial_{t,x}^\beta
a^*_2(\partial_t,\partial_x) \pi u_{2k+2}
\big\|_{L^2(\RR^d\times \TT^d)}\ =\
O(\langle t\rangle^k)\,.
$$
It is important that the
right hand side of the equation determining $a^*_2(\partial_t,\partial_x) \pi u_{2(k+1)}$ involves only derivatives of the earlier profiles and not
the profiles themselves.
The standard energy estimate for $a^*_2(\partial_t,\partial_x)$ implies that
for $\alpha\ne 0$,
$$
\big\|
\partial_{t,x}^\alpha \pi u_{2k+2}
  \big\|_{ L^2(\RR^d\times\TT^d) }
\ = \ O(\langle t\rangle^{k+1})\,.
$$
It remains to estimate $(I-\pi) u_n$ for $n=2k+2$ and $2k+3$.

Equation \eqref{eq:epsorthog} with index $2k+2$ in place of $k+2$
expresses $ (I-\pi) u_{2k+2}$ in terms of the profiles
with indices $\le 2k+1$. Those profiles are $O(\langle t\rangle^k)$
 by the inductive hypothesis. 
This yields
$$
\big\|
\partial_{t,x}^\alpha
(I-\pi)u_{2k+2 }
\big\|_{ L^2 (\RR^d\times\TT^d) }
\  = \ O( \langle t \rangle^k)
$$
an estimate stronger than the $O(\langle t\rangle^{k+1})$
required by the Theorem.

Equation \eqref{eq:epsorthog} with index $2k+3$
in place of $k+2$
expresses
$ (I-\pi)u_{2k+3}$
in terms of the profiles
with indices $\le 2k+2$.  Those with index $\le 2k+1$ are
$O(\langle t\rangle^k)$
by the inductive hypothesis.
All the derivatives of the profile $u_{2k+2}$
have just been shown to be $O(\langle t\rangle^{k+1})$.
It follows that
$$
\big\|
\partial_{t,x}^\alpha
(I-\pi)u_{2k+3}
\big\|_{L^2(\RR^d\times\TT^d) }
\ = \ O( \langle t \rangle^{k+1})\,.
\qquad
\Box
$$
\vskip.2cm

\section{High accuracy on $t\sim \eps^{-2+\delta}$ without crimes}
\label{sec:2minusdelta}

This section is devoted to a proof of the correctness of the traditional 
two scale ansatz for times strictly smaller than $\eps^{-2}$. 

\begin{theorem}
\label{thm:errorest}
For $k\in \N$, define a truncated ansatz, constructed from the
first non oscillating profiles
$\pi u_0, \pi u_2, \dots , \pi u_{2k}$ by
$$
U^k(\eps,t,x,y)\ :=\
\sum_{n=0}^{2k}\eps^n u_n(t,x,y)\ +\
 \eps^{2k+1} (I-\pi) u_{2k+1} 
 \ +\
\eps^{2k+2} (I-\pi) u_{2k+2}\,.
$$
 The 
approximate
solution is
$
U^k(\eps, t,x, x/\eps).
$
Denote by $u^\eps$ the exact solution of \eqref{eq:homprob}.
There is a constant $C$, independent of $0<\eps\le 1$  and $t\ge 0$, 
so  
the energy of the error is bounded by
\begin{equation}
\label{eq:errbd}
\big\| \nabla_{t,x} \big[u^\eps(t,x)\ -\   U^k(\eps, t,x, x/\eps)  \big]\big\|_{L^2(\RR^d_x)}
\ \le\
C\,
\eps^{2k+1} \langle t\rangle^{k+1}
\,.
\end{equation}
\end{theorem}

\begin{remark} 
\label{rem:elvis} 
 {\bf i.}
The energy of $u^\eps$ is bounded independently of time so the right hand
side of \eqref{eq:errbd} estimates the relative  energy error.

\noindent
{\bf ii.}  By choosing $k$
large one gets arbitrarily high order accuracy
on time intervals that grow as $1/\eps^{2-\delta}$ for any $\delta>0$.
{\rm Indeed, on the time interval $0\le t\le \eps^{-\gamma}$ with
$\gamma(k+1)<2k+1$ the relative error in energy is of order $\eps^{2k+1-\gamma(k+1)}$ 
and tends to zero as $\eps\to 0$.  }

\noindent
{\bf iii.}   The problem 
\eqref{eq:homprob} is invariant by differentiation in time.  The derivative
$\partial_t^ju^\eps$ is the solution of the same problem with source
term $\partial_t^jf$.  The profiles of the two scale asymptotic solution of that
problem are equal to the functions $\partial_t^ju_n(t,x,y)$.
Theorem \ref{thm:errorest} applied to that problem shows that with 
a constant $C$ depending on $j$ but independent of $t,\eps$,
\begin{equation}
\label{eq:errbd2}
\big\| \partial_t^j\nabla_{t,x} \big[u^\eps(t,x)\ -\   U^k(\eps, t,x, x/\eps)  \big]\big\|_{L^2(\RR^d)}
\ \le\
C\,
\eps^{2k+1} \langle t\rangle^{k+1}
\,.
\end{equation}

\end{remark}

The proof of Theorem \ref{thm:errorest} has three main ingredients.  The first
in \S \ref{sec:residual} relies on 
all the work done so far.  It  is a  precise formula for the
difference between $f$ and $\rho(x/\eps)\partial^2_t U^k-\dive\, a(x/\eps) \grad\, U^k$.
That difference has terms no more regular than $H^{-1}$.  
They are estimated in \S \ref{sec:reslem}.  The error
in energy with such singular source terms is bounded using
Proposition \ref{prop:energyh-1source}.

\subsection{Formula for the residual}
\label{sec:residual}

Suppose that  $U^k$ is  the finite power series from Theorem \ref{thm:errorest}. 
Then
\begin{equation}
\label{eq:fred}
 \bigg[
\rho(y) \partial_t^2  -  \mathcal{A}_{xx}
  - \frac{1} {\eps}
  \mathcal{A}_{xy}
  - \frac{1}{\eps^2}
  \mathcal{A}_{yy}
  \bigg]
  U^k(\eps,t,x,y)
  =
  W^k(\eps,t,x,y)
  =
  \sum_{n=-2}^{2k+2}
  \eps^n w_n(t,x,y).
\end{equation}
The profiles $w_j$ satisfy
$$
w_0=f,
\qquad
{\rm and\ for \ } j\ne 0, \ -2\leq j\le 2k,
\qquad
w_j\ =\ 0\, .
$$
Indeed, for $j\le 2k-2$ one has 
\begin{equation*}
 w_j=(\rho(y) \partial_t^2-\mathcal{A}_{xx})u_j-
 ( \mathcal{A}_{yy} u_{j+2}
 \ +\ 
 \mathcal{A}_{xy} u_{j+1} 
 )=0
\end{equation*}
by construction of the profiles $u_j$. For $j=2k-1$
use the fact that $\pi u_{2k+1}=0$ because 
$2k+1$ is odd to find that 
\begin{align*}
w_{2k-1}&=(\rho(y) \partial_t^2-\mathcal{A}_{xx}) u_{2k-1}-
 ( \mathcal{A}_{yy}(I-\pi) u_{2k+1}
 \ +\ 
 \mathcal{A}_{xy} u_{2k} )\\
&=(\rho(y) \partial_t^2-\mathcal{A}_{xx}) u_{2k-1}-
 ( \mathcal{A}_{yy} u_{2k+1}
 \ +\ 
 \mathcal{A}_{xy} u_{2k})=0.
\end{align*} 
 For $j=2k$
 use $\pi u_{2k+1}=0$ and $\mathcal{A}_{yy}\pi=0$ to find
\begin{align*}
w_{2k}&=(\rho(y) \partial_t^2-\mathcal{A}_{xx} ) u_{2k}-
 ( \mathcal{A}_{yy}(I-\pi) u_{2k+2}
 \ +\ 
 \mathcal{A}_{xy}(I-\pi) u_{2k+1})\\
&=(\rho(y) \partial_t^2-\mathcal{A}_{xx}) u_{2k}-
 ( \mathcal{A}_{yy} u_{2k+2}
 \ +\ 
 \mathcal{A}_{xy} u_{2k+1})=0.
\end{align*}

Therefore
\begin{align*}
 \bigg[
\rho(y) \partial_t^2\ &-\  \mathcal{A}_{xx}
  \ -\ \frac{1} {\eps}
  \mathcal{A}_{xy}
  \ -\ 
	\frac{1}
  {\eps^2}
  \mathcal{A}_{yy}
  \bigg]
  U^k(\eps,t,x,y)\ - f
  \cr
\ & =\
  \eps^{2k+1} w_{2k+1}(t,x,y) \ +\ \eps^{2k+2} w_{2k+2}(t,x,y)
	    \ =:\  r(\eps, t,x,y)\,.
\end{align*}
and thus
$$
\Big[\rho(x/\eps)\partial_t^2\ -\ 
  \dive \, a(x/\eps) \, \grad 
\Big]
U^k(\eps,t,x,x/\eps) \ -\ f
\ =\
r(\eps,t,x,x/\eps)\,.
$$

Equation \eqref{eq:fred} shows that
\begin{align*}
w_{2k+1}&=(\rho(y) \partial_t^2 -\mathcal{A}_{xx})(I-\pi)u_{2k+1}-
  \mathcal{A}_{xy}(I-\pi) u_{2k+2} ,
 \cr
w_{2k+2}&=(\rho(y) \partial_t^2-\mathcal{A}_{xx}) (I-\pi)u_{2k+2}.
\end{align*}
Therefore 
$$
r=
   \Big[ \rho(y) \partial_t^2-
    \mathcal{A}_{xx}
  \Big](I-\pi)
  \Big(
  \eps^{2k+1} u_{2k+1} + \eps^{2k+2} u_{2k+2}
  \Big)
  \ -\
   \eps^{2k+1}
  \mathcal{A}_{xy}(I-\pi)
u_{2k+2}\,.
$$
The definitions of the operators $\mathcal{A}$ yield
\begin{align}
\begin{split}
\label{eq:residualnoncriminal}
  r=&\ \eps^{2k+1}\left[\left(\rho(y) \partial_t^2-
    \div_x\, a(y)\grad_x\right)
  (I-\pi)
   u_{2k+1} + \eps\rho(y)\partial^2_t(I-\pi) u_{2k+2}\right]\\
	&-\eps^{2k+1}\left(\div_x\, a(y)\grad_y + \div_y\,a(y)\grad_x + \eps\div_x\,a(y)\grad_x\right) (I-\pi)u_{2k+2}\\
	:=&\ \eps^{2k+1}\left( I(\eps,t,x,y) \ +\  II(\eps,t,x,y)\right).
	\end{split}
\end{align}
In addition, for $\ell =2k+1,2k+2$,  with $\chi_n$ given by
\eqref{eq:chikcalpha},
\begin{equation}
\label{eq:I-piul}
(I-\pi)u_\ell \ =\
\sum_{n=1, \ \ell-n\  {\rm even}}^\ell
\,
\chi_n\
\pi u_{\ell-n}
= 
\sum_{n=1, \ \ell-n\  {\rm even}}^\ell
\,
\sum_{|\beta|=n}
c_{\beta,n}(y)\,\partial_{t,x}^\beta\pi u_{\ell-n}
\,.
\end{equation}

\subsection{Estimates for the residual}
\label{sec:reslem}

In view of \eqref{eq:I-piul} and Theorem \ref{thm:secular},
$I(\eps,t,x,y)$ in \eqref{eq:residualnoncriminal} 
is a sum of terms of the form $\eps^p c(y)v(t,x)$ with $p\geq 0$,  
$c\in L^2(\TT^d)$ and
\begin{equation*}
\|\partial_{t,x}^\alpha v\|_{L^2(\RR^d)}
\ \lesssim\
\langle t\rangle^k\
\end{equation*}
for $\alpha\in \N^{1+d}$, including $\alpha=0$. 

Proposition \ref{prop:reslemnew}
implies that  the $L^2$-norm of $I(\eps,t,x,x/\eps)$ 
satisfies 
\begin{equation}
\label{eq:IepsIIeps}
\|I(\eps, t,x,x/\eps)\|_{L^2(\RR^d_x)}
\ \lesssim\
 \langle t\rangle^k\,. 
\end{equation}
The  term $II(\eps,t,x,y)$ in \eqref{eq:residualnoncriminal} involves derivatives
 of $a(\cdot)$,
so is not square integrable.   Equation  \eqref{eq:I-piul}
shows that $II(\eps,t,x,y)$ is equal to
\begin{align*}
-\left(\div_x\, a(y)\grad_y + \div_y\,a(y)\grad_x + \eps\div_x\,a(y)\grad_x\right) 
\bigg[\sum_{n=1}^{2k+2}\sum_{|\beta|=n}c_{\beta,n}(y)\partial^\beta_{t,x}\pi u_{2k+2-n}(t,x)\bigg].
\end{align*}
Evaluate at $y=x/\eps$ to show that, with $\div$ acting on functions depending on $x$ as well as on functions depending on $x/\eps$,
\begin{align}
\begin{split}
\label{eq:calcIIIeps}
II(\eps,t,x,x/\eps)=-\eps&\sum_{n=1}^{2k+2}\sum_{|\beta|=n}\div\left[a(x/\eps)c_{\beta,n}(x/\eps)\grad\,\partial^\beta_{t,x}\pi u_{2k+2-n}(t,x)\right]\\
+&\sum_{n=1}^{2k+2}\sum_{|\beta|=n}\left(a(x/\eps)(\grad_y c_{\beta,n})(x/\eps)\cdot\grad\,\partial^\beta_{t,x}\pi u_{2k+2-n}(t,x)\right)\\
:=&\ \div\left(II^{(1)}(\eps,t,x,x/\eps)\right) \ +\  II^{(2)}(\eps,t,x,x/\eps).
\end{split}
\end{align}
Arguing exactly as for $I$,
using that each $c_{\beta,n}\in H^1(\TT^d)$, it follows that $II^{(1)}(\eps,t,x,x/\eps)$, $\partial_t II^{(1)}(\eps,t,x,x/\eps)$, 
and, $II^{(2)}(\eps,t,x,x/\eps)$ are in $L^2(\RR^d)$ with 
\begin{align}
\begin{split}
\label{eq:IIIeps}
\|\partial_t II^{(1)}(\eps,t,x,x/\eps)\|_{L^2(\RR^d_x)}
\ +\
\|II^{(1)}(\eps,t,x,x/\eps)\|_{L^2(\RR^d_x)}&
\ \lesssim\
\eps\langle t \rangle^k\\
\quad\|II^{(2)}(\eps,t,x,x/\eps)\|_{L^2(\RR^d_x)}&
\ \lesssim\
\langle t \rangle^k.
\end{split}
\end{align}
Combining estimate \eqref{eq:IIIeps} with \eqref{eq:IepsIIeps} 
yields the residual 
\begin{equation}
\label{eq:len}
r(\eps,t,x,x/\eps)\ =\ f(\eps, t,x) \ +\  \div\, g(\eps, t,x)\,,
\end{equation}
with $f(\eps, t,x)=\eps^{2k+1}(I+II^{(2)})(\eps,t,x,x/\eps)$ and $g(\eps,t,x)=\eps^{2k+1}II^{(1)}(\eps,t,x,x/\eps)$ and the two estimates,
\begin{equation}
\label{eq:fgestimate}
\|f(\eps,t,\cdot)\|_{L^2(\RR^d)}\lesssim\eps^{2k+1}\langle t\rangle^k,
\quad
\|\partial_t g(\eps, t,\cdot)\|_{L^2(\RR^d)}+\|g(\eps, t,\cdot)\|_{L^2(\RR^d)}\lesssim\eps^{2k+2}\langle t\rangle^k.
\end{equation}

\subsection{End of proof of Theorem \ref{thm:errorest}}

Denote by $u^\eps$ the exact solution and $U^k$ the approximation
from the statement of Theorem \ref{thm:errorest}. We have proved that
$$
\Big[\rho(x/\eps)\partial_t^2\ -\ 
  \dive \, a(x/\eps) \, \grad
\Big](u^\eps(t,x)-U^k(\eps,t,x,x/\eps))\ :=\ r(\eps, t,x,x/\eps)
$$
with $r(\eps,t,x,x/\eps)$ 
satisfying \eqref{eq:len} and 
\eqref{eq:fgestimate}.
Apply Proposition \ref{prop:energyh-1source}.
In the estimate of that Proposition, the $L^1\big([0,t]\,;L^2(\RR^d) \big)$ norms of 
$f$ and $\partial_t g$ are estimated by $t$ times the 
$L^\infty\big([0,t];L^2(\RR^d)  \big)$ norms, which are controlled by \eqref{eq:fgestimate}. 
This yields the claimed result 
\eqref{eq:errbd}.
\hfill
\qed
\vskip.2cm

\begin{remark}  The results concerning the two scale expansions extend 
with only minor changes in the proofs  to the case of coefficients $a(x,x/\eps)$
provided that for all $\beta$, $\partial_x^\beta a(x,y)\in L^\infty(\RR^d\times\TT^d)$.
In this case the correctors have coefficients $c_\alpha(x,y)$ 
and Lemma \ref{lem:wellposedchi} asserts
$\partial_x^\beta c_\alpha\in  L^\infty(\RR^d_x;H^1(\TT^d))$.  The proofs of the leap
frog structure and slow secular growth are unchanged.   The
residual estimate for the term $II$ in \eqref{eq:calcIIIeps} 
has a few additional terms treated using this regularity of $c$. 
For the criminal path, the case $a(x,x/\eps)$ is work in progress.
\end{remark}

\section{The criminal path}
\label{sec.crim}

The criminal path, briefly presented in Section \ref{sec.crimintro},
yields approximations valid for times as long as $\eps^{-N}$ for arbitrary $N$.

{\bf Main idea.}  
{\it The criminal path changes the  choice of the 
nonoscillatory parts $\pi u_n$.     
The oscillatory parts 
$(1-\pi) u_n$ are given in terms of 
the nonoscillatory parts by
\eqref{eq:osc} as in classical homogenization.}

We replace the traditional ansatz \eqref{eq:utaylor} 
for $U$ by the criminal ansatz \eqref{eq:crimeansatz} for $V$. 
Since the terms $v_n$ in \eqref{eq:crimeansatz} depend on $\eps$, 
we commit the asymptotic crime of mixing
different orders in $\eps$.  
The 
 terms of order $\eps^2$ in the
criminal path are introduced in different but related 
ways in the 
seminal articles
  \cite{BP}, 
  \cite{santosa},
  \cite{Lam}.

\subsection{Derivation of criminal equations}
According to Definition \ref{def:criminal} the criminal ansatz 
satisfies $v_0=\pi v_0$, $\pi v_n=0$ for $n\geq1$ and 
$(I-\pi) v_n=\chi_n v_0$.  The leading term $v_0=\pi\,v_0$ and
profile
$V(\eps, t,x,y)$ are constructed so that the two formal identities 
\begin{equation}
\label{eq:epsD}
\Big(
a_2^*(\eps\partial_{t,x})+ \cdots +
a_{2n}^*(\eps\partial_{t,x}) + \cdots
\Big)
v_0(\eps,t,x) = \eps^2 f\,,
\qquad
v_0 = 0 \mbox{ for } t<0\,,
\end{equation}
\begin{equation}
\label{eq.bp2}
V(\eps,t,x,y) \ =\  \Big(1+
\sum_{l= 1}^\infty \eps^l \chi_l(y,\partial_{t,x})\ \Big) v_0(\eps,t,x) \,,
\end{equation}
are satisfied up to an acceptable error.
Even if \eqref{eq:epsD} is truncated to 
be a finite sum, it  is high order in $t$.
For each $\eps$ it 
usually defines an ill posed time evolution.
In spite of this, the next sections construct functions $v_0^k$ that satisfy \eqref{eq:epsD}  with 
small enough error.

Equation \eqref{eq:epsD} can be understood in another way.
Theorems \ref{thm:osc} and \ref{thm:recur1} together with 
their remarks show that 
standard homogenization hierarchy  is equivalent 
to the pair  of identities in the sense of formal
power series,
\begin{equation*}
\Big(
\sum_{n= 1}^\infty
a_{2n}^*(\eps\partial_{t,x})
\Big)
\Big(\sum_{m=0}^\infty \eps^m \pi u_m\Big) = \eps^2 f,
\quad
U = \Big(1+
\sum_{l= 1}^\infty \eps^l \chi_l\ \Big) \Big(\sum_{n=0}^\infty \eps^n \pi u_n\Big).
\end{equation*}
Equivalently
\begin{equation}
\label{eq.bp}
\Big(
\sum_{n= 1}^\infty
a_{2n}^*(\eps\partial_{t,x})
\Big)
\pi U=\eps^2f\,,
\qquad
U = \Big(1+
\sum_{l= 1}^\infty \eps^l \chi_l\ \Big) \pi U.
\end{equation}
{\it If one does {\bf not} insist that $\pi U$ be  a formal power 
series in $\eps$}, this suggests the criminal equation \eqref{eq:epsD} for $v_0=\pi U$ 
and the criminal ansatz $V=U= \Big(1+\sum_{n=1}^\infty \eps^n \chi_n\ \Big) v_0$.

In the next sections, equation \eqref{eq:epsD} is converted to normal form, truncated at order $k$ and filtered, 
leading to the solutions $v_0^k$ from Definition \ref{def:criminal} 
with small enough error such that the approximation is very accurate.

\subsection{Elimination algorithms}
\label{sec:elimination}

The algorithms of this section eliminate
the time derivatives in 
\eqref{eq:epsD}, other than those in $a_2^*$,
while
changing the $x$-derivatives in the high order terms.

\begin{proposition}
\label{prop:elim2}
There are uniquely determined homogeneous operators $R_{2j}(\partial_{t,x})$
and $\widetilde a_{2j}(\partial_x)$
of degree $2j$, the latter involving only $\partial_x$, 
so that
\eqref{eq:R2j} holds as an identity in the sense of formal power series.
\end{proposition}

The heart of the proof is the following Lemma.
\begin{lemma}
\label{lem:elim}
Suppose that $m\ge 2$ and $S_{2m}(\partial_{t,x})$ 
is homogeneous
of degree $2m$ and contains only even powers of $\partial_t$. 
Then there exists a unique $r_{2m-2}(\partial_{t,x})$, 
homogeneous of degree $2m-2$,  
so that $r_{2m-2}a_2^*+S_{2m}$ is a differential operator in $\partial_x$ only. 
\end{lemma}

{\bf Proof.}   
Write
$$
r_{2m-2}(\partial_{t,x})
\ =\ q_0 \partial_t^{2m-2} + q_2(\partial_x)  \partial_t^{2m-4}
+\
\cdots\ +
q_{2m-4}(\partial_x)  \partial_t^{2}
 +
q_{2m-2}(\partial_{x})\,.
$$
The goal is to determine $q_0,\dots,q_{2m-2}$ in such a way that $r_{2m-2}a_2^*+S_{2m}$ is a differential operator in $\partial_x$ only. 
Order the terms in $S_{2m}$ 
according to the order of the time derivative
$$
S_{2m}\ =\ s_0(\partial_x) \partial_t^{2m} + s_2(\partial_x)  \partial_t^{2m-2}
+\
\cdots\ +
s_{2m-2}(\partial_x)  \partial_t^{2}
\,.
$$
Define $\urho:=\pi \rho$ and $a_2(\partial_x)$
so that $a_2^*$ from \eqref{eq:Tom} satisfies
\begin{equation}
\label{eq:JP}
a_2^*(\partial_{t,x}) = \urho\partial_t^2 + a_2(\partial_x)\,.
\end{equation}
In particular $a_2(\partial_x)$ is second order in $\partial_x$.  
Then the terms containing time derivatives in $r_{2m-2}a_2^*$
are equal to 
\begin{align}
\begin{split}
\label{eq:ra2star}
&\urho
\Big(
q_0 \partial_t^{2m} + q_2(\partial_x)  \partial_t^{2m-2}
+\
\cdots\ +
q_{2m-4}(\partial_x)  \partial_t^{4}
 +
q_{2m-2}(\partial_{x})\partial_t^2
\Big)\cr
& +\ 
\Bigl(
(q_0 a_2)(\partial_x)\partial_t^{2m-2} + (q_2 a_2)(\partial_x)\partial_t^{2m-4}+\
\cdots\ + (q_{2m-6} a_2)(\partial_x)\partial_t^{4} + (q_{2m-4} a_2)(\partial_x)\partial_t^2 
\Bigr).
\end{split}
\end{align}
Regrouping in order of decreasing powers of $\partial_t$ yields that \eqref{eq:ra2star} equals
\begin{align*}
\urho q_0\, \partial_t^{2m}\ +\ 
\Bigl(\urho q_2 + q_0 a_2\Bigr)  (\partial_x)\,\partial_t^{2m-2} &+\
\cdots\cr + 
\Bigl(\urho q_{2m-4} + q_{2m-6}a_2\Bigr)(\partial_x)\,\partial_t^4
& +\  
\Bigl(\urho q_{2m-2} + q_{2m-4}a_2\Bigr)(\partial_x)\,\partial_t^2.
\end{align*}
The unique choice eliminating the time derivatives in $r_{2m-2}a_2^*+S_{2m}$ is 
given by
$$
q_0=-\urho^{-1}s_0,
\quad
{\rm
and\ for\ } 1\leq j\leq m-1,
\quad
q_{2j}(\partial_x)=-\urho^{-1}\Bigl(s_{2j}(\partial_x) + (q_{2j-2}a_2)(\partial_x)\Bigr).
$$
This completes the proof of Lemma \ref{lem:elim}.
\hfill
\qed
\vskip.2cm

\begin{definition}
Denote by $\caO_N$ 
the set of constant coefficient
partial differential operators in $\partial_{t,x}$ that are 
sums of terms  homogeneous of degree 
at least $N$. 
\end{definition}

The operators in $\caO_N$ are those whose
symbols vanish to order $N$ at the origin.

{\bf Proof of Proposition  \ref{prop:elim2}. }
In the next expressions $\caO_N$ represents an element of 
$\caO_N$. 
The identity  \eqref{eq:R2j} holds if and only if for all $k\ge 2$
\begin{equation}
\label{eq:eliminationorderk}
\Big(
1+R_2+R_4+\cdots + R_{2k-2}
\Big)
\Big(a_2^* + a_4^* +\cdots+a_{2k}^*
\Big)
=
a_2^* + \widetilde a_4 +\cdots + \widetilde a_{2k} + \caO_{2k+2}\,.
\end{equation}

The goal is to find $R_{2j}$ such that \eqref{eq:eliminationorderk} holds. 
For $k=2$
 expanding yields
\begin{equation}
\Big(
1+R_2
\Big)
\Big(a_2^* + a_4^* 
\Big)
=a_2^* + R_2 a_2^* + a_4^* + \caO_{6}.
\end{equation}
The  term of order $4$ 
is $R_2 a_2^* + a_4^*$.  Choose $R_2$ using  Lemma
\ref{lem:elim}
as the  unique homogeneous order 2 operator so that this
fourth order term is independent of $\partial_t$.  Denote by
$\widetilde a_4$ that differential operator.  

The construction is recursive. Suppose that
the $R_2,\dots, R_{2k-2}$ and $\widetilde a_4,\cdots, \widetilde a_{2k}$
have been uniquely determined so that 
\eqref{eq:eliminationorderk} holds. We show that $R_{2k}$ and 
$\widetilde a_{2k+2}$ are uniquely  determined so that the 
case $k+1$ of 
\eqref{eq:eliminationorderk}
is satisfied.

In the case $k+1$ of \eqref{eq:eliminationorderk} the terms of order $\le 2k$ on the right are only influenced
by $R_2, \dots, R_{2k-2}$.  
Separating the lowest order term in $\caO_{2k+2}$ the right hand side of \eqref{eq:eliminationorderk} can be written as 
\begin{equation}
\label{eq:inductivestep}
a_2^* + \widetilde a_4 +\cdots + \widetilde a_{2k} + p_{2k+2} + \caO_{2k+4}, 
\end{equation}
where $p_{2k+2}(\partial_{t,x})$ is homogeneous of degree $2k+2$. 
To prove the case $k+1$ one must determine $R_{2k}$ and $\widetilde a_{2k+2}$ such that
\begin{equation}
\label{eq:finitenormal}
\Big(
1+R_2+R_4+\cdots + R_{2k-2}+R_{2k}
\Big)
\Big(a_2^* + a_4^* +\cdots+a_{2k+2}^*
\Big)
=
a_2^* + \widetilde a_4 +\cdots + \widetilde a_{2k+2} + \caO_{2k+4}\,.
\end{equation}
The term of order $2k+2$ 
is $R_{2k}a_2^* + p_{2k+2}$, where $p_{2k+2}$ is given by \eqref{eq:inductivestep}, in terms of the 
$R_{2j}, \widetilde a_{2j}$ that are known from the inductive step.
Lemma \ref{lem:elim} shows that there is a unique  $R_{2k}$ so that 
this term of order $2k+2$  is independent of  $\partial_t$.  That is the uniquely
determined $R_{2k}$ and the operator in $\partial_x$ is $\widetilde a_{2k+2}$.
The recursive construction is complete.
\hfill
\qed
\vskip.2cm

\begin{remark}   The proof
yields  a recursive 
algorithm to compute
$R_{2j}, \widetilde a_{2j}\,$ from the $a_{2j}^*$.  The computation of the 
coefficients of $a_{2j}^*$ requires the solution of $\sim d^{2j}$
cell problems.
\end{remark}

\begin{remark}
\label{rem.inverse}
Proposition 
\ref{prop:elim2} implies 
 that if
$$
a_2^*(\partial_{t,x}) 
+ a_4^*(\partial_{t,x}) + \cdots + a_{2k}^*(\partial_{t,x}) 
 =E(\partial_{t,x}) 
$$
with $E(\partial_{t,x}) \in 
 \caO_{2k+2}$, then there is a
 $\widetilde E(\partial_{t,x})\in\caO_{2k+2}$ so that 
$$
a_2^*(\partial_{t,x}) 
+ \widetilde a_4(\partial_{x}) + \cdots + \widetilde a_{2k}(\partial_{x}) 
 =\widetilde E(\partial_{t,x})\,.
$$
\end{remark}

The converse of Remark \ref{rem.inverse} is also true.  In the  ring of formal
power series $1+\sum_{j=1}^\infty R_{2j}(\partial_{t,x})$
has a unique multiplicative inverse
$$
1+ \sum_{j=1}^\infty\widetilde R_{2j}(\partial_{t,x})
\ =\ 
1+\sum_{k=1}^\infty
\Big(
-\sum_{j=1}^\infty R_{2j}(\partial_{t,x})
\Big)^k\,,
$$
satisfying
\begin{equation*}
\Big(
1+ \sum_{j=1}^\infty\widetilde R_{2j}(\partial_{t,x})
\Big)
\Big(
1+\sum_{j=1}^\infty R_{2j}(\partial_{t,x})
\Big)
=
\Big(
1+\sum_{j=1}^\infty R_{2j}(\partial_{t,x})
\Big)
\Big(
1+ \sum_{j=1}^\infty\widetilde R_{2j}(\partial_{t,x})
\Big)
 = 1.
\end{equation*}

The next Corollary is an immediate consequence.

\begin{corollary}  
\label{cor:RtildeR}
Define 
$
 \widetilde R^k(\partial_{t,x})
 \, :=\, \sum_{j=1}^{k}
 \widetilde R_{2j}(\partial_{t,x})
 \,.
 $
   Then
   \begin{equation}
\begin{aligned}
\label{eq:inversereduction}
\big(
1+\widetilde R^k(\partial_{t,x})
\big)
\Big[ 
a_2^*(\partial_{t,x}) 
&+ \widetilde a_4(\partial_{x}) + \cdots + \widetilde a_{2k+2}(\partial_{x}) 
\Big]
\cr
\ =\ 
&
\Big[
a_2^*(\partial_{t,x}) 
+ a_4^*(\partial_{t,x}) + \cdots + a_{2k+2}^*(\partial_{t,x}) 
\Big]
+
\caO_{2k+4}\,.
\end{aligned}
\end{equation}
\end{corollary}

\subsection{Criminal equation with time derivatives eliminated} 

Having constructed the operators $R_{2j}$ the elimination algorithm is used to transform 
equation \eqref{eq:epsD} to the normal form justifying \eqref{eq:disp5}. Then equation \eqref{eq:disp5} 
is truncated at order $k$ yielding equation 
\eqref{eq:disp} repeated here,
\begin{equation*}
\Big[
a^*_2(\partial_t, \partial_x)
 \ +\
 \sum_{j=2}^{k+1}
 \eps^{2j-2}\
 \widetilde a_{2j}(\partial_x)
 \Big]
 \uv_0^k(\eps,t,x) =
 \Big[
 1+
 R^k\big(\eps\partial_{t,x}\big)\Big]f,
 \quad
 \uv_0^k=0 \ \ {\rm for}\ \ t<0\,.
\end{equation*}

The operator in brackets on the left frequently defines
 an ill posed time evolution. 
 This instability does not doom the construction
of good approximations.  
Committing an error, which high order in $\eps$, we filter the source term 
$[1+R^k\big(\eps\partial_{t,x}\big)]f$. 

Choose cutoff functions $\psi_j\in C^\infty_0(\RR^d)$ for $j=1,2$ 
with $\psi_1= 1$ on a neighborhood of 
the origin and $\psi_2=1$ on a neighborhood of
${\rm supp}\, \psi_1$. 
Choose $0<\alpha<1$.  We compute a profile
$v_0^k$ that 
 satisfies
 with
 $D:=(1/i)\partial_x$,
\begin{equation}
\label{eq:six}
\begin{aligned}
\big[
a_2^*(\partial_t,\partial_x) + \eps^2 \widetilde a_4(\partial_{x}) 
 +
\cdots
 + 
\eps^{2k}\widetilde a_{2k+2}(\partial_x)
\big] 
v_0^k  =\psi_1(\eps^\alpha D)\big(1+R^{k}(\eps\partial_{t,x}) \big)f.
\end{aligned}
\end{equation}
The ill posed evolutions remain. However, for the filtered sources on the right
in 
\eqref{eq:six},
there exist nice solutions.
 Fourier transformation in $x$ yields  ordinary differential
 equations in time 
  parametrized by $\xi$
 for any tempered solution
 of \eqref{eq:six}.
 It shows
 that  a tempered 
solution must have  transform with  support in $\eps^{-\alpha}\{ {\rm supp}\, \psi_1\}$.
Such a solution satisfies
$\psi_2(\eps^\alpha D)v_0^k =v_0^k$.  
Therefore it also satisfies
 \begin{equation}
\label{eq:seven}
\begin{aligned}
\Big[
a_2^*(\partial_t,\partial_x) + 
\big[\eps^2 \widetilde a_4(\partial_{x}) 
 +
\cdots
 + 
\eps^{2k}\widetilde a_{2k+2}(\partial_x)
\big]
\psi_2(\eps^\alpha D)
\Big] 
v_0^k
= \psi_1(\eps^\alpha D)\big(1+R^{k}(\eps\partial_{t,x}) \big) f.
\end{aligned}
\end{equation}

\subsection{Stability Theorem}
\label{sec:stability}

The operator applied to $v_0^k$ in \eqref{eq:seven} is,
$$
a^*_2(\partial_t,\partial_x) \ +\ M(\eps, k, \partial_{x})\,,
\qquad
M(\eps,k, \partial_{x})\ :=\ 
\sum_{j=2}^{k+1}
\ 
\eps^{2j-2}\
\widetilde a_{2j}(\partial_{x})\ \psi_2(\eps^\alpha  D)\,,
\quad
0<\alpha<1\,.
$$
The operator $M$ also depends on $\alpha$ and  $\psi_2$.

\begin{theorem}
\label{thm:stable}  
There is an $\eps_0>0$ so that
for each $\eps\le \eps_0, 0<\alpha<1, k\in\N$, 
and $g_0,g_1\, \in \, H^1(\RR^d)\times L^2(\RR^d)$ there is a unique solution
$v$ with $\partial_t^jv\in C(\RR\,;\, H^{1-j}(\RR^d))$ for $j\ge 0$ to
$$
\big[
a_2^*(\partial_{t,x}) \ +\ M(\eps,k,\partial_{x})
\big]v\ =\ 0\,,
\qquad
v(0,\cdot) \ =\ g_0\,,
\quad
\partial_t v(0,\cdot)\ =\ g_1\,.
$$
This solution satisfies with a constant $C$ independent of $\eps$
and $v$,
$$
\sup_{t\in\RR}
\Big(
\|\nabla_x v(t)\|_{L^2(\RR^d)}
\ +\ 
\|\partial_t v(t)\|_{L^2(\RR^d)}
\Big)
\ \le\ 
C\Big(
\|\nabla_x v(0)\|_{L^2(\RR^d)}
\ +\ 
\|\partial_t v(0)\|_{L^2(\RR^d)}
\Big)\,.
$$
For any $\alpha_0<1$, the bound is uniform for $0<\alpha\le \alpha_0$.
\end{theorem}

\begin{remark}
Duhamel's principle implies that there is a constant $C$ so that  the 
unique
tempered solution of 
$$
\Big[
a_2^*(\partial_{t,x})  \ +\ M(\eps,k,\partial_{x})
\Big]v\ =\ f\,,
\qquad
v(0,\cdot) \ =\ 
\partial_t v(0,\cdot)\ =\ 0\,,
$$
satisfies for all $t,f$,
$
\|\nabla_{t,x} v(t)\|_{  L^2(\RR^d)  }
\ \le\ 
C\,
\|f\|_{L^1([0,t];L^2(\RR^d) ) }
$.
\end{remark}

\noindent
{\bf Proof of Theorem \ref{thm:stable}.}   
$M$ is bounded from $H^s\to H^{s+\sigma}$ for all $s, \sigma$ with bound 
independent of $s$.  The bound tends to infinity  as $\eps\to 0$.  
The boundedness implies the existence statement of the Theorem.
That the solutions are bounded independently of $\eps$ is more subtle. This is shown below.

With the notation from \eqref{eq:JP},
the equation $(a_2^*+M)v=0$
has the form $\urho v_{tt}=-\mu(D)v$ with 
\begin{equation}
\label{eq:defrho}
-\mu_\eps(\xi) \ :=\ 
a_2(i\xi)
\ +\ 
\psi_2(\eps^\alpha\xi)
\sum_{2\le j\le k+1}
\eps^{2j-2}\,
\widetilde a_{2j}(i\xi)
\,.
\end{equation}
Each summand on the right is real valued.
Theorem \ref{thm:stable}  follows from the following estimate.  For each $\alpha_0\in ]0,1[$
there is an $\eps_0>0$  and 
constants $0<c<C$ so that for all $0<\eps\le \eps_0$, $0<\alpha\le \alpha_0$ and all $\xi\in \RR^d$,
\begin{equation}
\label{eq:rholower}
c\,|\xi|^2
 \ \le\ \mu_\eps(\xi)\ \le\ C\,|\xi|^2\,.
\end{equation}

To prove
\eqref{eq:rholower} it suffices to show that the modulus of the second summand 
on the right hand side of  \eqref{eq:defrho}
is much smaller than the modulus of the first.  In the support of the second summand
$|\xi|\lesssim \eps^{-\alpha}$.  For such $\xi$ it holds that
$$
\big|
\eps^{2j-2} \widetilde a_{2j}(i\xi)
\big| 
\ \lesssim\
 \eps^{2j-2} |\xi|^{2j} 
 \ =\ (\eps|\xi|)^{2j-2}|\xi|^2
\ \lesssim\ \eps^{(1-\alpha)(2j-2)}|\xi|^2.
 $$
 The first factor on the right tends to zero
 as $\eps\to 0$, since $0<\alpha<1$ and $j\geq 2$. This proves the desired inequality.

The spatial Fourier transform of the solution satisfies
$$
\urho\ 
\frac{\partial^2  \widehat v }{\partial^2 t}
\ +\ 
\mu_\eps(\xi)\,\widehat v\ =\ 0\,.
$$
Multiplying by  the complex conjugate of $\partial_t \widehat v$ and taking the
real part proves the conservation laws
$$
\forall \xi\in \RR^d,
\qquad
\frac{\partial}{\partial t}
\bigg(
\frac{\urho\, | \partial_t \widehat v(t,\xi) |^2}{2}
\ +\
\frac{{\mu_\eps(\xi)} \ |\widehat v(t,\xi)|^2}{2}
\bigg)
\ =\ 0\,.
$$
The estimate \eqref{eq:rholower} implies that the conserved quantity
$$
\int_{\RR^d} \left(
\urho\,
\big|
\partial_t \widehat v (t,\xi)
\big|^2
\ +\
\mu_\eps(\xi) \ 
\big|
\widehat v(t,\xi)
\big|^2 \right)
\ d\xi
$$
is uniformly  equivalent to 
$
\| \partial_t v(t)\|_{L^2(\RR^d)}^2
+
\|\nabla_x v(t)\|_{L^2(\RR^d)}^2 $.   This completes the proof.
\hfill
\qed
\vskip.2cm

\begin{corollary}
\label{cor:hier}  Let $0<\alpha_0<1,k\in \N$ and $f\in H^\infty(\RR\times\RR^d)$ supported in $[0,1]\times\RR^d$. 
Then for and all $0<\alpha\le \alpha_0$ and $0< \eps <\eps_0$
there is a unique solution $\zeta\in C^\infty(\RR\,;\, H^\infty(\RR^d))$ of 
\eqref{eq:six}. It  satisfies ${\rm supp}\,  \widehat \zeta\subset
\eps^{-\alpha}\,{\rm supp}\, \psi_1$ and 
\begin{equation}
\sup_{t\in [0,\infty[}
\big\|
\partial_{t,x}^\beta
\zeta
\big\|_{L^2(\RR^d)}
\ \le\
C(k,f,\beta)
\ <\ \infty\,,
\qquad
\beta\ne 0\,,
\end{equation}
\begin{equation}
\big\|
\zeta(t)
\big\|_{L^2(\RR^d)}
\ \le\
C(k,f) \,
\langle t\rangle
\,.
\end{equation}
\end{corollary}

{\bf Proof.}  {\sl Uniqueness.}  
Taking the Fourier transform shows that any solution $\zeta\in C^\infty(\RR\,;\,H^\infty(\RR^d))$
to \eqref{eq:six} satisfies 
${\rm supp}\, \widehat \zeta\subset
\eps^{-\alpha}\,{\rm supp}\, \psi_1$.  Therefore $\zeta$ also
satisfies \eqref{eq:seven}.  The solutions of those equations
are uniquely determined thanks to Theorem \ref{thm:stable}.

{\sl Existence.}  
Define $\zeta$ as the solution to 
\eqref{eq:seven}. 
The same Fourier transform argument as for uniqueness shows
that this function satisfies 
${\rm supp}\, \widehat \zeta\subset
\eps^{-\alpha}\,{\rm supp}\, \psi_1$.   In particular
$\psi_2(\eps^\alpha D) \zeta=\zeta$.
This implies that $\zeta$ satisfies
\eqref{eq:six}. 
 Theorem \ref{thm:stable}
implies that it has the additional properties claimed in 
Corollary \ref{cor:hier}
establishing existence.
\hfill
\qed
\vskip.2cm

\subsection{Criminal approximation error}
\label{ss.crim-error}

This section performs the computations that are the main
ingredients in the proof of Theorem \ref{thm:criminal}.  

\begin{remark}
The profile  $v_0^k=\pi v_0^k$ from 
Definition \ref{def:criminal}
satisfies $\psi_2(\eps^\alpha D) v_0^k=v_0^k$
so is the unique tempered solution supported in $t\ge 0$ of
\begin{equation}
\label{eq:nopsi}
\Big[
a_2^*(\partial_{t,x} )
\ +\
\Big(\sum_{j=2}^{k+1}
\eps^{2j-2}
\widetilde a_{2j}(\partial_{x})
\Big)
\psi_2(\eps^\alpha D) 
\Big]
\pi v_0^k
\ =\
\big[
1+ R^{k}(\eps \partial_{t,x})
\big]
\psi_1(
\eps^\alpha D)
f\,.
\end{equation}
\end{remark}

 \S \ref{sec:crimeres} 
 computes a precise
 formula
for the residual.   
  The entire
paper prepares that computation.

\subsubsection{Formula for  the residual}
\label{sec:crimeres}
Recall the criminal approximation $V^k$ from \eqref{eq:defVk}, 
\begin{equation*}
V^k(\eps,t,x,y)\,
=\,\sum_{n=0}^{2k+2}\eps^n v_n^k(\eps,t,x,y)
=\left(I + \sum_{n=1}^{2k+2}\eps^n\chi_n(y,\partial_t,\partial_x)\right)v_0^k(\eps,t,x)\, .
\end{equation*}
Define
\begin{equation}
Z(\eps, t,x,y)
\ :=\
 \bigg[
\rho(y) \partial_t^2\ -\ 
  \frac{1}
  {\eps^2}
  \mathcal{A}_{yy}
  \ -\
   \frac{1}
  {\eps}
  \mathcal{A}_{xy}
  \ -\
  \mathcal{A}_{xx}
  \bigg]
V^k(\eps,t,x,y)\,.
\end{equation}
This is regrouped in powers of $\eps$ as if the $v_n^k$ did not depend on
$\eps$.\footnote{That is, the computations are
made
in the ring of Laurent expansions in  $\eps$ whose coefficients are
functions
of $\eps,t,x,y$.
In $\eps^n v_n^k$,
the function $v_n^k$ is a coefficient of
$\eps^n$. 
If for instance $v_n^k=\eps^2$
 the power from $v_n^k$ must not
be combined with the $\eps^n$, the expression $\eps^nv_n^k$  
is still a term in $\eps^n$.}
This yields
\begin{equation}
Z \ =\ \sum_{j=-2}^{2k+2} \eps^j \ Z_j
\ =\
\sum_{j=-2}^{2k} \eps^j \ Z_j
\ +\
\sum_{j=2k+1}^{2k+2} \eps^j \ Z_j
\ =:\
\sum_{j=-2}^{2k} \eps^j \ Z_j
\ +\
\caE_1\,.
\end{equation}
The term $\caE_1$ is the first  error term.  It is estimated in Lemma \ref{lem:etimatecaE1}.
Theorem \ref{thm:osc} shows that the definition of the nonoscillatory
parts of the $v^k_n$ is equivalent to
\begin{equation}
\big(
I-\pi\big)
Z_n=0\,,
\qquad
-2\le n\le 2k.
\end{equation}

Since $\pi v^k_n=0$ for $1\le n\le 2k+2$ and $v_0^k=\pi v_0^k$, Theorem \ref{thm:recur1} yields
\begin{equation*}
\pi \sum_{j=-2}^{2k} \eps^j  Z_j
\ =\ \sum_{j=0}^{2k} \eps^j a_{j+2}^*(\partial_{t,x})v^k_0
=\eps^{-2}\sum_{j=0}^{2k} \eps^j a_{j+2}^*(\eps\partial_{t,x})v^k_0\,.
\end{equation*}
Equation \eqref{eq:inversereduction} of Corollary \ref{cor:RtildeR} implies
\begin{align*}
\pi \sum_{j=-2}^{2k} \eps^j  Z_j
\ &=\
\eps^{-2}
\big(1+\widetilde R^{k}(\eps\partial_{t,x})\big)\Big[
a_2^*(\eps\partial_{t,x} )
\ +\
\sum_{n=2}^{k+1}
\widetilde a_{2n}(\eps\partial_{x})
\Big] v^k_0
+\eps^{-2}\caO_{2k+4}(\eps\partial_{t,x}) v^k_0,
\cr
\ &=\ 
\big(1+\widetilde R^{k}(\eps\partial_{t,x})\big)\Big[
a_2^*(\partial_{t,x} )
\ +\
\sum_{n=2}^{k+1}\eps^{2n-2}
\widetilde a_{2n}(\partial_{x})
\Big] v^k_0
+\eps^{-2}\caO_{2k+4}(\eps\partial_{t,x}) v^k_0.
\end{align*}
Since $v_0^k$ satisfies equation \eqref{eq:six}, 
\begin{equation*}
 \begin{aligned}
 \pi \sum_{j=-2}^{2k} \eps^j  Z_j
\ &=\ 
\big(1+\widetilde R^{k}(\eps\partial_{t,x})\big)\psi_1(\eps^\alpha D)\big(1+
R^{k}(\eps\partial_{t,x})\big) f
+\eps^{-2}
\caO_{2k+4}(\eps\partial_{t,x}) v^k_0
\cr
\ &=:\
\big(1+\widetilde R^{k}(\eps\partial_{t,x})\big)\big(1+
R^{k}(\eps\partial_{t,x})\big)\psi_1(\eps^\alpha D) f +\caE_2
\,.
 \end{aligned}
 \end{equation*}
 Use $(1+\widetilde R^{k}(\partial_{t,x}))(1+ R^{k}(\partial_{t,x})) 
 = 1 + \caO_{2k+2}$ to continue the computation, 
  \begin{align*}
 \big(1+\widetilde R^{k}(\eps\partial_{t,x})\big)
   \big(1+
R^{k}(\eps\partial_{t,x})\big)
\psi_1(\eps^\alpha D) f
\   &=\
  \big(1+ \caO_{2k+2}(\eps\partial_{t,x}))
  \psi_1(\eps^\alpha D)f
  \cr
  &=:\
   \psi_1(\eps^\alpha D)f  +\caE_3
   \cr
   &=\
   f + (\psi_1(\eps^\alpha D)-1)f + \caE_3
   \cr
    &=:
   f + \caE_4+ \caE_3.
  \end{align*}
  Therefore,
 \begin{equation}
 \label{eq:resid}
 Z(\eps,t,x,y)-f(t,x) \ =\
 \caE_1(\eps, t,x,y)
 \ +\
 \sum_{j=2}^4 \caE_j(\eps, t,x)\,.
 \end{equation}
with
\begin{align}
\begin{split}
\label{eq:listoferrorterms}
\caE_1 &= \eps^{2k+1} Z_{2k+1} + \eps^{2k+2} Z_{2k+2},
\qquad
\caE_2=
\eps^{-2}
\caO_{2k+4}(\eps\partial_{t,x}) v^k_0,
\cr
\caE_3
&=
\caO_{2k+2}(\eps\partial_{t,x})
\,  \psi_1(\eps^\alpha D)f,
  \qquad
\hskip.3cm  \caE_4  =
(\psi_1(\eps^\alpha D)-1)f .
\end{split}
\end{align}

\subsection{Residual estimates and proof of Theorem \ref{thm:criminal}}

The error  $u^\eps(t,x)-V^k(\eps,t,x,x/\eps)$ satisfies 
\begin{equation}
\label{eq:florian1}
\begin{aligned}
\Big[\rho(x/\eps)\partial_t^2 -\dive\, a(x/\eps)\, \grad\Big] \Big(
u^\eps(t,x)&-V^k(\eps,t,x,x/\eps)\Big)
 \cr
 &=\ \caE_1(\eps,t,x,x/\eps) \ +\  \sum_{j=2}^4 \caE_j(\eps,t,x).
\end{aligned}
\end{equation}

\begin{lemma}
\label{lem:etimatecaE1}
The error term $\caE_1(\eps,t,x,x/\eps)$ from \eqref{eq:listoferrorterms} is of the form 
\begin{equation*}
\caE_1(\eps, t,x,x/\eps)
\ =\
f(\eps,t,x) 
\ +\
 \div\, g(\eps, t,x)
\end{equation*}
with $f,g$ satisfying
uniformly in $t\ge0$,
\begin{align}
\label{eq:fgestimate2}
\|f(\eps, t,\cdot)\|_{L^2(\RR^d)}
\ +\ 
\|\partial_t g(\eps, t,\cdot)\|_{L^2(\RR^d)}
\ +\
\|g(\eps, t,\cdot)\|_{L^2(\RR^d)}
\ \lesssim\
\eps^{2k+1}.
\end{align}
\end{lemma}
\begin{proof}
As for the residual in the non criminal approximation 
write
\begin{align*}
\mathcal{E}_1(\eps, t,x,y)&=
 \eps^{2k+1} Z_{2k+1}(t,x,y) \ +\ \eps^{2k+2} Z_{2k+2}(t,x,y)\\
	&=\Big[ \rho(y) \partial_t^2-
    \mathcal{A}_{xx}
  \Big](I-\pi)
  \Big(
  \eps^{2k+1} v^k_{2k+1} + \eps^{2k+2} v^k_{2k+2}
  \Big)
  \ -\
   \eps^{2k+1}
  \mathcal{A}_{xy}(I-\pi)
v^k_{2k+2}\,.
\end{align*}
Since $(I-\pi)v^k_j=\chi_j\pi v^k_0=$ for $1\leq j\leq 2k+2$ and $\chi_j=\sum_{|\beta|=j}c_{\beta,j}(y)\partial^\beta_{t,x}$, it follows that
$$
\mathcal{E}_1(\eps, t,x,y)\ =\ 
\eps^{2k+1} \left(I(\eps,t,x,y)\ +\
II(\eps,t,x,y)\right),
$$
where
\begin{align*}
I(\eps,t,x,y) \ :=\ \big(\rho(y) \partial_t^2-
    \div_x\, a(y)\,\grad_x
  \big)  &\sum_{|\beta|=2k+1}  c_{\beta,2k+1}(y)\partial^\beta_{t,x}\pi v^k_0(\eps,x,t) 
\cr
+ \,\eps\,\rho(y)\,  \partial^2_t   &\sum_{|\beta|=2k+2}c_{\beta,2k+2}(y)\partial^\beta_{t,x}\pi v^k_0(\eps, t,x),
\end{align*}
\begin{align*}
II(\eps,t,x,y)\, :=\,
- \Big(\div_x a(y)\grad_y + \div_y  & a(y)\grad_x + 
\eps\,\div_xa(y)\grad_x\Big) 
\cr
&
\times
\Big(
\sum_{|\beta|=2k+2}c_{\beta,2k+2}(y)\partial^\beta_{t,x}\pi v^k_0(\eps,t,x)
\Big).
\end{align*} 

Use that $c_{\beta,2k+1}, c_{\beta,2k+2}\in H^1(\TT^d_y)$  to show that
$I(\eps, t,x,y)$ is a sum of terms of the form $\eps^p c(y)v(t,x)$ with $p\geq 0$, $c\in L^2(\TT^d)$.
Corollary \ref{cor:hier}
 implies that each $v$ satisfies, uniformly in  $t\ge 0$,
\begin{equation*}
\|\partial^\alpha_{t,x}v(t,x)\|_{L^2(\RR^d)}
\ \lesssim\
 1
\end{equation*}
for $\alpha\in\N^{1+d}$, including $\alpha=0$. 
Proposition \ref{prop:reslemnew} in the appendix
implies that uniformly for $t\ge  0$,
\begin{equation*}
\|I(\eps, t,x,x/\eps)\|_{L^2(\RR_x^d)}
\ \lesssim\
 1.
\end{equation*}
The second term, $II(\eps, t,x,y)$, involves derivatives of $a(\cdot)$. 
As in \eqref{eq:calcIIIeps}, write
\begin{align*}
II(\eps,t,x,x/\eps)
\ &=\ -\eps\, 
 \div\Big[
 a(x/\eps)
 \sum_{|\beta|=2k+2}
 c_{\beta, 2k+2}(x/\eps)\grad\,
 \partial^\beta_{t,x} v^k_0(\eps,t,x)
 \Big]
 \\
&\qquad + 
a(x/\eps)
\sum_{|\beta|=2k+2}\Big(
\big(\grad_y c_{\beta, 2k+2}\big)
(x/\eps)\cdot\grad\,\partial^\beta_{t,x} v^k_0(\eps,t,x)
\Big)
\\
&:=  \,\div\big(
II^{(1)}( \eps,t,x,x/\eps)\big) \ +\ 
 II^{(2)}(\eps,t,x,x/\eps)\,.
\end{align*} 
As for $I$ it follows that, 
uniformly in $t\ge 0$,
\begin{align*}
\|II^{(1)}(\eps,t,x,x/\eps)\|_{L^2(\RR^d_x)} \ +\  
\|\partial_t II^{(1)}(\eps,t,x,x/\eps)\|_{L^2(\RR^d_x)}&\ \lesssim\
 \eps,\\
\|II^{(2)}(\eps,t,x,x/\eps)\|_{L^2(\RR^d_x)}&
\ \lesssim\
 1.
\end{align*}
Defining $f(\eps,t,x):=\eps^{2k+1}(I + II^{(2)})(\eps,t,x,x/\eps)$ and $g(\eps,t,x):=\eps^{2k+1}II^{(1)}(\eps,t,x,x/\eps)$ completes the  proof of 
Lemma 
\ref{lem:etimatecaE1}.
\end{proof}
{\bf End of proof of Theorem \ref{thm:criminal}.}
Estimate the error terms $\caE_2,\caE_3,\caE_4$ from \eqref{eq:listoferrorterms}. 
Since $v^k_0$ and all of its derivatives are uniformly bounded in $L^2(\RR_x^d)$, 
one has
\begin{equation}
\label{eq:one}
\big\|
\caE_2(\eps, t,x)
\big\|_{L^2(\RR^d)}
\ +\
\big\|
\caE_3(\eps, t,x)
\big\|_{L^2(\RR^d)} \ \lesssim\ \eps^{2k+2}\,.
\end{equation}
The error from $\caE_4$ is smaller.  For any $N$ one has
\begin{equation}
\label{eq:two}
\big\|
\caE_4(\eps, t,x)
\big\|_{L^2(\RR^d)} \ \lesssim\ \eps^{N}\,.
\end{equation}

Theorem \ref{thm:criminal} is a 
consequence of Lemma \ref{lem:etimatecaE1}, 
\eqref{eq:one},
\eqref{eq:two},
and
Proposition \ref{prop:reslemnew} in the appendix.
\hfill
$\Box$
\vskip.3cm

\begin{remark}
In the same way as in part {\bf iii} of Remark \ref{rem:elvis}, one finds a constant $C$ 
depending on $j$ but independent of $\eps\le 1,t\ge 0$ so that
\begin{equation}
\big\|
\partial_t^j\nabla_{t,x} \big(u^\eps(t) \ -\ V^k(\eps, t,x,x/\eps)\big)
\big\|_{L^2(\RR^d_x)}
\ \le \ C\, \eps^{2k+1}\, \langle t\rangle
\,.
\end{equation}

\end{remark}

\section{Sources growing polynomially in time}
\label{sec.polynom}

The error estimates
 for sources with compact support in $0<t<1$
 (Theorem \ref{thm:errorest} 
for the classical case and 
Theorem \ref{thm:criminal} 
in the criminal case)
easily imply similar estimates
\eqref{eq:ClassicLong}
and 
\eqref{eq:CriminalLong}
for sources 
 that grow at most
polynomially in time. Suppose that 
$$
\exists m, \ \forall \alpha,\ 
\exists C,\
\forall t\,,\
\quad
\big\|
\partial_{t,x}^\alpha f(t)
\big\|_{L^2(\RR^d_x)}
\ \le\ 
C\, t^m,
\qquad
f=0\ \   {\rm for}\ \ t\le 0\,.
$$
Use a partition of unity
to write $f=\sum_{j=1}^\infty f_j$ with 
$f_j$ supported in $[j-1,j]$
and 
$$
\exists C,\
\forall \alpha, j, \
\forall  t,
\quad
\sup_{t\in \RR} \big\|\
\partial_{t,x}^\alpha f_j(t)
\big\|_{L^2(\RR^d_x)}
\ \le\
C\
\sup_{j-1\le t\le j}\
\big\|
\partial_{t,x}^\alpha f(t)
\big\|_{L^2(\RR^d_x)}\,.
$$
Denote by $u^\eps_j$ the exact solution with right hand side $f_j$ and by 
$u^\eps_{j, \rm approx}$ either the classical or the criminal approximation.
Then $u^\eps=\sum u_j^\eps$ and 
$u_{\rm approx}^\eps=\sum u_{j, \rm approx}^\eps$.
With
$
e^\eps_j(t)  := 
\big\|
\nabla_{t,x}\big[
u^\eps_j(t)
-
u^\eps_{j, \rm approx}(t)
\big]
\big\|_{L^2(\RR^d_x)}$, 
the triangle inequality implies that the error in energy satisfies
$$
\big\|
\nabla_{t,x}\big[
u^\eps(t)
-
u^\eps_{\rm approx}(t)
\big]
\big\|_{L^2(\RR^d_x)}
\ \le\ 
\sum_{j-1\le t} e^\eps_j(t)
\,.
$$

\subsection{Classical homogenization}

Fix $k$ the index in the classical approximation.  
Apply Theorem \ref{thm:errorest} 
 with initial time shifted to $j-1$.  
Since $f_j$ has size $O(j^m)$, and the error in 
 Theorem \ref{thm:errorest}  depends linearly on $f$,
 \begin{equation*}
\exists C,\ 
\forall j,t,\qquad
e_j^\eps(t)
\ \le \ C \, \eps^{2k+1}\  \big\langle (t-(j-1))_+ \big\rangle^{k+1}
\ j^m\,.
\end{equation*}

This implies
\begin{equation}
\label{eq:ClassicLong}
\big\|
\nabla_{t,x}\big[
u^\eps(t)
-
u^\eps_{\rm approx}(t)
\big]
\big\|_{L^2(\RR^d_x)}
\ \lesssim\
\eps^{2k+1}
\
\int_0^t 
\langle (t-s)_+ \rangle^{k+1}
\ s^m\
ds
\ \lesssim \ 
\eps^{2k+1}
\
\langle t\rangle^{k+2+m}\,.
\end{equation}
Given $N>0$ and $0<\delta< 2$, choosing
$k$ so large that 
$2k+1\ge N+ (k+2+m)(2-\delta)$ guarantees that
 the total error is 
$O(\eps^N)$ for times
$t\le  1/\eps^{2-\delta}$.

\subsection{Criminal path}

Theorem \ref{thm:criminal} shifted in time yields 
$e_j^\eps(t) \lesssim \eps^{2k+1}\, \big\langle (t-(j-1))_+\big\rangle\, j^m$, so
\begin{equation}
\label{eq:CriminalLong}
\big\|
\nabla_{t,x}\big[
u^\eps(t)
-
u^\eps_{\rm approx}(t)
\big]
\big\|_{L^2(\RR^d_x)}
\ \lesssim\
\eps^{2k+1}\
\int_0^t 
\,
\big\langle (t-s)_+\big\rangle
\
s^m\ ds
\ \lesssim\
\eps^{2k+1}
\ \langle t\rangle^{m+2}\,.
\end{equation}
Choosing 
$k$ so large  that 
$2k+1\ge N +N(m+2)$
 shows that the total error is  
$O(\eps^N)$  on intervals
$t\le \eps^{-N}$.

\section{Systems and Schr\"odinger equation}
\label{sec.system}

\subsection{Second order systems}
This section considers systems
of wave equations including  the elastodynamics equations. 
 The unknown $u(t,x):\RR^{d+1}\to \RR^p$ is $\RR^p$ valued.
   Its gradient is a function with values in 
   $\caM_{p,d}  $ the set of  $p\times d$ matrices.
 The scalar product of matrices is
   defined by
  multiplying corresponding components and summing.   
 The systems
have the form
$$
\rho(x/\eps) \partial^2_t u \ -\ \dive\, a(x/\eps)\, \grad\, u \ =\ f(t,x),
\qquad
\rho, a 
\quad
{\rm periodic},
\quad
u=f=0\ \ {\rm for}\ \ t<0\,,
$$
with source term $f:\RR^{1+d}\to \RR^p$. 
Assume that for  ${\rm a.e.}\ y\in \TT^d$ the fourth order tensor $a(y)$ is 
a symmetric map from $\caM_{p,d}$ to itself.
Assume that $a\in L^\infty(\TT^d)$ is 
uniformly positive definite in the sense that 
$$
\exists m_1>0,
\ 
\forall \eta\in \caM_{p,d},
\qquad
a(y) \eta \cdot \eta \ \geq\ m_1 |\eta|^2  
\qquad
{\rm  a.e.} \  y \in \TT^d\,.
$$
Assume  that 
 $\rho(y)\in L^\infty(\TT^d ;\caM_{p,p})$ is  symmetric 
and uniformly positive definite. 

Keep the notations  $\caA_{yy}$, $\caA_{xy}$, $\caA_{xx}$.  These operators now
map  $\RR^p$ valued functions to $\RR^p$ valued
functions. 
They are $p\times p$ matrices of operators. 
The operator $\Pi$ is defined for vector-valued functions as the component wise average over $\TT^d$. 
The operators $\chi_k(y,\partial_{t,x})$ in the next definition
have coefficients that belong to
$H^1(\TT^d;\caM_{p,p})$. The definition is analogous to the scalar case in Definition \ref{def.chi-k}.  
\begin{definition}
Set $\chi_{-1} :=  0$ and $\chi_0 :=  I$, where $I$ is the $p\times p$ 
identity matrix. 
For $k\ge 2$, define operators mapping $\RR^p$-valued functions of 
$t,x$ to $\RR^p$-valued functions of $t,x,y$ by
\begin{equation}
\begin{aligned}
\label{eq:rec1}
 \chi_k(y,\partial_{t,x})  \ :=\  - \caA_{yy}^{-1} (I-\pi) 
\Big[ \caA_{x y} \chi_{k-1}\ +\  (\caA_{xx}-\rho\partial_t^2) \chi_{k-2} \Big] 
.
\end{aligned}
\end{equation}
In the above equation, the composition rule
uses $p\times p$ matrix multiplication. 
\end{definition}

\begin{remark}
Denote by $\bfe_j$ the canonical basis vectors of $\RR^p$.
The product of 
the coefficient of $\partial_{t,x}^\alpha$ in $\chi_k$ 
and 
$\bfe_j$ yields a vector whose components in $H^1(\TT^d)$ 
 are the  usual correctors 
 with index  $\alpha$
 from periodic 
homogenization.
\end{remark}

\begin{definition}
For $n\geq2$
define the $p\times p$ system  of differential
operators $a^*_n$  by,
\begin{equation}
\label{eq:rec33}
a^*_n(\partial_t,\partial_x) \ = \ \pi \Big( 
\caA_{xy} \chi_{n-1}(y,\partial_{t,x}) \ +\ (\caA_{xx}-\partial_t^2\rho) 
\chi_{n-2}(y,\partial_{t,x}) \Big) . 
\end{equation}
Then $a^*_n(\partial_t,\partial_x)$ is  homogeneous  of degree $n$ 
in $(\partial_t,\partial_x)$.  It  contains only even powers of $\partial_t$. 
\end{definition}

\begin{theorem}
\label{thm:rec1}
The cascade of equations for the system wave equation is equivalent to
two families of equations.   The first
\begin{equation}
\label{eq:rec2}
- \sum_{n=2}^{k+2} 
a^*_n(\partial_t,\partial_x)\, \pi \, u_{k+2-n}  \ = \  f \, \delta_{0k} 
\end{equation}
yields wave equations defining  the nonoscillatory parts
in terms of the source and earlier nonoscillatory profiles.
The second
\begin{equation}
\begin{aligned}
\label{eq:rec2bis}
 (I-\pi) u_k  \ = \ \sum_{n=1}^k \chi_n(\partial_t,\partial_x,y) \,\pi \, u_{k-n} (\tau,\xi)\,,
\end{aligned}
\end{equation}
gives the oscillatory  parts in terms of nonoscillatory  parts.
\end{theorem}

\begin{theorem}
\label{thm:rec2}
For any $n\geq1$, the coefficients of $a^*_n(\partial_t,\partial_x)$ are symmetric
 $p\times p$ matrices. 
For any odd $n\geq1$, the homogenized operator of order $n$ vanishes.
That is  for $m\geq1$,
 $a^*_{2m+1}(\partial_t,\partial_x)=0$.
\end{theorem}

{\bf Proof.}
Introduce the operators $C_1, C_2, D_1, D_2, W_k, Z_k$ as in the proof of 
Theorem \ref{thm:recur2}. Their coefficients are now $p\times p$ matrices
of operators. 
For given vectors $\lambda,\mu \in\RR^p$, since $\chi_0=I$, definition (\ref{eq:rec33}) 
implies that
$$
\begin{aligned}
a^*_k(&\partial_t,\partial_x)\lambda\cdot\mu
\cr & = \int_{\T^d} \Big( \caA_{\xi y} W_{k-1}(\tau,\xi) \chi_0(y) \lambda
+ (\caA_{\xi\xi}-\tau^2\rho) W_{k-2}(\tau,\xi) \chi_0(y)\lambda \Big) \cdot \chi_0(y)\mu \, dy 
\cr
& = \int_{\T^d} \Big(  W_{k-1}(\tau,\xi) \chi_0(y)\lambda\cdot (-\caA_{\xi y} \chi_0(y)\mu)
+ W_{k-2}(\tau,\xi) \chi_0(y)\lambda\cdot ((\caA_{\xi\xi}-\tau^2\rho) \chi_0(y)\mu) \Big)  dy 
\cr
& = (-1)^k \int_{\T^d} \chi_0(y)\lambda\cdot \Big( Z_{k-1}(\tau,\xi) \caA_{\xi y} \chi_0(y) \mu 
+ Z_{k-2}(\tau,\xi) (\caA_{\xi\xi}-\tau^2\rho) \chi_0(y) \mu \Big) dy \, .
\end{aligned}
$$
Integrating  by parts  as 
in the proof of Theorem \ref{thm:recur2} using analogous adjoint relations
yields
\begin{align*}
a^*_k(\partial_t,\partial_x)\lambda\cdot\mu &= (-1)^k \int_{\T^d} \chi_0(y)\lambda\cdot \Big[ \caA_{\xi y} W_{k-1}(\tau,\xi) \chi_0(y) \mu
+ (\caA_{\xi\xi}-\tau^2\rho) W_{k-2}(\tau,\xi) \chi_0(y) \mu \Big] dy
\cr 
&= (-1)^k \lambda\cdot a^*_k(\partial_t,\partial_x)\mu \, .
\end{align*}
For odd $k$  this yields $a_k^*=0$, and, for even $k$, the coefficients
of $a_k^*$ are symmetric. 
\hfill
$\Box$
\vskip.3cm

\subsection{Schr\"odinger  equation}

Consider
the homogenization of Schr\"odinger's equation
 $$
i\,\rho(x/\eps) \, \partial_t u^\eps \, - \, \big(\dive \, a(x/\eps)\, \grad\big)\, u^\eps
\ =\
 f(t,x)\,.
 $$
With only the most minor modifications of the proof
one finds 
an anologue of Theorem \ref{thm:recur2}.
This yields a leap frog structure
and slow secular growth for the two scale expansions.
The analogue of Proposition \ref{prop:energyh-1source}
is that if $g\in L^\infty(\RR;L^2(\RR^d)) \cap L^1(\RR;L^2(\RR^d))$ vanishes for $t\le 0$
and $g_t\in L^1(\RR;L^2(\RR^d))$, then 
the solution of $i \,\rho(x/\eps) \,u_t = \dive (a(x/\eps) \grad u )+ \dive g $
vanishing for $t\le 0$ is uniformly bounded in $L^\infty(\RR;H^1(\RR^d))$.
{The supplementary material to this paper contains a  
proof of this
estimate.}
The classical approximation is accurate for times
$\sim 1/\eps^{-2+\delta}$ for any small $\delta>0$.  The criminal approach
yields approximations with error $\le \eps^N$
for times  $1/\eps^N$.

\appendix
\section{Examples with   maximal secular growth}
\label{sec:maxgrowth}

\subsection{Saturated secular growth}
Theorem \ref{thm:secular} gives an upper bound on the growth in time of the profiles.  
This appendix  shows that the upper bound is attained for 
generic problems in dimension $d=1$. Consider 
$$
\partial_t^2 u \ -\ \partial_x\,\big( a(x/\eps)\, \partial_x u\big) \ =\ f(t,x)\ \in C^\infty_0(]0,1[\times \RR)\,.
$$
with $a(y)$ 1-periodic and not identically constant. Denote
$$
a^*_2(\partial_t,\partial_x) \ :=\ 
\partial_t^2
\ -\ 
c^2\partial_x^2
\ =\ 
\big(
\partial_x -c\partial_t
\big)
\big(
\partial_x +c\partial_t
\big),
$$
the homogenized operator. Then $u_0=u_0(t,x)$ is the solution of 
$a^*_2(\partial_t,\partial_x) u_0=f$ that vanishes for $t<0$.  Therefore for $t>1$ there are
uniquely determined $g_0,h_0\in C^\infty(\RR)$ with $g_0(s)=0$ for $s\gg 1$
and $h_0(s)=0$ for $s\ll -1$ so that for $t>1$ one has
$$
u_0 \ =\ g_0(x-ct) \ +\ h_0(x+ ct)\,.
$$
For most $f$, both $g_0$ and $h_0$ are not identically equal to zero.  This is
true in particular if $f\ge 0$ and not identically equal to zero.   In that
case both $g_0$ and $h_0$ are non negative and not identically zero.
The profile $\pi u_2$ satisfies
of
$$
a^*_2(\partial_t,\partial_x) \pi u_2 \ =\ -a_4^*(\partial_t, \partial_x) u_0\,.
$$
The fourth order homogeneous polynomial
$-a_4^*(\tau,\xi)$ contains no odd powers of $\tau$.  
For $t>1$,  $u_0$ satisfies
$a_2^*(\partial_t,\partial_x)u_0=0$.  In this domain,
replacing systematically $\partial_t^2u_0$ by $c^2\partial_x^2 u_0$ and  
equivalently $\tau^2$ by $c^2\xi^2$ yields
a new polynomial $q(\xi)$ so that 
$-a_4^*(\partial_t,\partial_x)u_0 = q(\partial_x) u_0$.   
As soon as $a(y)$ is not  constant one has
$q(\partial_x) = \gamma \partial_x^4$ with $\gamma\ne 0$ 
(see \cite{COV2}). This shows that
$$
-\,a_4^*(\pm c, 1) \ =\ \gamma \ \ne 0\,.
$$
The equation for $\pi u_2$ for $t>1$ is
\begin{equation}
\label{eq:dal1}
a^*_2(\partial_t,\partial_x) \pi u_2 \ =\ 
(\gamma \partial^4g_0)(x-ct) + (\gamma \partial^4h_0)(x+ c t)
\,.
\end{equation}
Writing $a_2^*(\partial_t,\partial_x) = (\partial_t -c\partial_x)(\partial_t +c\partial_x)$ one 
verifies that the function
$$
z_2(t,x) = \frac{1}{4c^2}\,\Big[(ct+x)\,
(\gamma \partial^3g_0)(x-ct) \ +\ 
(ct-x)\,(\gamma \partial^3h_0)(x+ct)\Big]\,.
$$
satisfies \eqref{eq:dal1}. Choose a cutoff function $\chi\in C^\infty(\RR_t)$ equal to zero
for $t<1/2$ and equal to $1$ for $t\ge 1$.   Then 
$$
a^*_2(\partial_t,\partial_x)\Big( \pi u_2
\ -\ \chi(t) z_2 \Big)
$$
is compactly supported in $0\le t\le1$.  Therefore
$$
\pi u_2\, =\,\chi(t)\, z_2(t,x) + r_2(t,x),
\quad
{\rm and,}
\quad
\forall  \alpha,  \ 
\partial_{t,x}^\alpha r_2\in L^\infty(\RR^{1+1}).
$$
The equation for $\pi u_4$  is
$$
a^*_2(\partial_t,\partial_x)\pi u_4 \ =\ 
-\,a_4^*(\partial_t, \partial_x) \pi u_2\, -\, a_6^*(\partial_t ,\partial_x) \pi u_0\,.
 $$
Only the $\pi u_2$ term on the right is unbounded.   Furthermore, if any of the derivatives
in $a_4^*(\partial_t,\partial_x)$ fall on the factors $ct\pm x$ in $\pi u_2$ the resulting function is 
bounded.  Therefore
\begin{align*}
a_4^*(   &  \partial_t,\partial_x) \Big[
(ct   +x)\,
(\gamma \partial^3g_0)(x-ct) 
+
(ct-x)\,(\gamma \partial^3h_0)(x    +ct)
\Big]
\cr
&=\
(ct+x)\,
(\gamma^2 \partial^7g_0)(x-ct) 
+
(ct-x)
(\gamma^2 \partial^7h_0)(x+ct)\ +\  L^\infty(\RR^{1+1} ).
\end{align*}
Reasoning as above yields  with 
$
\partial_{t,x}^\alpha r_3\in L^\infty(\RR^{1+1})
$,
$$
\pi u_4=
 \frac1{(4c^2)^2}\,\Big[\frac{(ct+x)^2}{2}
((\gamma\partial^3)^2
g_0)(x-ct) +
\frac{(ct-x)^2}{2}
((\gamma\partial^3)^2
h_0(x+ct)\Big]
\ +\
 \langle t\rangle 
r_3\,.
$$
An induction yields
$$
\pi u_{2n} \ =\ 
 \frac{(ct+x)^n}{n!}
g_n(x-ct) \ +\ 
\frac{(ct-x)^n}{n!} h_n(x+ct)
\ +\
 \langle t\rangle^{n-1} r_n\,,
$$
$$
g_n\, :=\, \Big(
\frac{\gamma\partial^3}{4c^2}\Big)^n
g_0\,,
\qquad
h_n\, :=\, \Big(
\frac{\gamma\partial^3}{4c^2}\Big)^n
h_0\,,
\qquad
\partial_{t,x}^\alpha r_n\in L^\infty(\RR^{1+1})\,.
$$
The $u_n$ saturate the upper bounds of
 Theorem  \ref{thm:secular}.

In addition note  that $u_{2n}$ grows with $n$ as the $(3n)^{\rm th}$
derivative of $g_0,h_0$.  This implies that  for generic real analytic $g_0,h_0$,
the series $\sum \eps^n u_n$ is  divergent.

\subsection{Classic homogenization is inaccurate beyond $t=\eps^{-2}$}  
\label{sec:beyond}
One could imagine
that by including  many  correctors, the classical algorithm might 
be accurate for times beyond $\eps^{-2}$.  The example of the preceding  
subsection show that that is not the case. 
For that example the exact solution satisfies $\sup_{0<\eps<1}\sup_{\RR^{1+1}}|u|<\infty$.   For $c>0$, $\delta>0$ as small as one likes define $t_\eps=c/\eps^{2+\delta}$.
To show the inaccuracy of the classical approximation it suffices to show that
for any $0<N\in \ZZ$, 
\begin{equation}
\label{eq:arnold3}
\lim_{\eps\to 0}\
\Big\|
\sum_{j=0}^{2N}
\eps^j \,
u_j(t_\eps)
\Big\|_{L^\infty(\RR_x)} \ =\ \infty\,.
\end{equation}

For the example
one has
for $t$ large 
\begin{equation}
\label{eq:arnold1}
\eps^{2j}\, t^{j}
\ \lesssim\
\|\eps^{2j} \pi u_{2j}(t)\|_{L^\infty(\RR_x)}
\ \lesssim\
\eps^{2j}\, t^{j}
\,.
\end{equation}
For $N$ fixed, formula \eqref{eq:radwanska} implies that 
\begin{equation}
\label{eq:arnold2}
\|\eps^{k} (I-\pi )u_{k}(t)\|_{L^\infty(\RR_x)}
\ \lesssim\
\eps^{k-1}\, t^{k-1}
\,.
\end{equation}
  Therefore, with $\lambda_\eps:=\eps^{-\delta}$ one has
$$
\Big\|
\eps^{2N} (I-\pi) u_{2N}(t_\eps)+\sum_{k=0}^{2N-1} \eps^{k}  u_{k}(t_\eps)
\Big\|_{L^\infty(\RR_x)}
 \lesssim 
\lambda_\eps^{2N-1}
\ \ {\rm and}\ \
\lambda_\eps^{2N}
\lesssim
\big\|
\eps^{2N} \pi u_{2N} (t_\eps)
\big\|_{L^\infty(\RR_x)}.
$$
It follows that with $C_j>0$,
$$
\Big\|
\sum_{j=0}^{2N}
\eps^j \,
u_j(t_\eps)
\Big\|_{L^\infty(\RR_x)} 
\ \ge\ 
C_1\,\lambda_\eps^{2N}
\ -\ 
C_2\,\lambda_\eps^{2N-1}\,.
$$
The limit $\eps\to 0$ yields
\eqref{eq:arnold3}.

\section{Two scale $L^2$ estimate}
\label{sec.estimate}

This appendix contains a proof of a classical estimate for  
oscillating two scale  functions.  It  is used in the error estimates
in Sections \ref{sec:reslem} and
\ref{ss.crim-error}.

\begin{proposition}
\label{prop:reslemnew}  
For each integer $s>d/2$, 
there is a constant $C$ so that for all
 $v\in H^s(\R^d)$ and $c\in L^2(\TT^d)$, 
 \begin{equation}
\label{eq:residlemmanew}
\int_{\RR^d}
\  |\, v(x)\, c(x/\eps)\,|^2\,\ dx
\ \le\
C\,
\|c\|^2_{L^2(\TT^d)}\,
\sum_{|\alpha| \le s}
  \int_{\RR^d}
  |
  (\eps\partial_x)^\alpha v(x)|^2\  dx
\,.
\end{equation}
\end{proposition}

{\bf Proof.} 
Denote by $Y:=[0,1)^d$ the unit box. 
Then  $\RR^d$ is  a disjoint union of   boxes $Y_k:= k+ Y$,
$\RR^d=\bigcup_{k\in\Z^n} Y_k$.
Scaling by $\eps$ yields
$\RR^d=\bigcup_{k\in\Z^n} \eps\,Y_k$.
Sobolev's inequality for $Y_k$ reads
\begin{equation}
\label{eq:sobolevnew}
 \big\|
 w
 \big\|_{L^\infty(Y_k)}^2
 \ \leq 
 C(s,d)\sum_{|\alpha| \le s}
  \int_{Y_k}
  |\partial^\alpha_y w(y)|^2\  dy\,,
  \qquad
  w\in H^s(Y_k)\,.
\end{equation}
When $y\in Y_k$, $x:=\eps y\in \eps Y_k$.
For $v\in H^s(\eps Y_k)$,
apply \eqref{eq:sobolevnew} 
 to  $w(y) := v( \eps y )\in H^s(Y_k)$ to find
\begin{equation}
 \label{eq:epsBnew}
 \| v\|_{L^\infty(\eps Y_k)}^2 
 \leq C(s,d) \,
 \eps^{-d}
 \sum_{|\alpha| \le s}
  \int_{\eps Y_k}
  |
  (\eps\partial_x)^\alpha v(x)|^2\  dx\,.
  \end{equation}
  Estimate, using \eqref{eq:epsBnew} in the last line,
  \begin{align*}
  \int_{\eps Y_k}
  \left|
   v(x)\   c(x/\eps)
   \right|^2\ dx
  \ &\le \
	\big\|
  v
  \big\|^2_{L^\infty(\eps Y_k)}
   \int_{\eps Y_k} \,
  |c(x/\eps)|^2
  \ dx \
  \cr
  \ &\leq   
  \ C(s,d)\,
\eps^{-d} \,   \sum_{|\alpha| \le s}
 \int_{\eps Y_k}
  |
  (\eps\partial_x)^\alpha v(x)|^2\  dx\,\eps^d\,\big\|
  c\big\|_{L^2(\TT^d)}^2\ .
\end{align*}
 Summing over $k$ yields 
 \eqref{eq:residlemmanew}.
 \hfill
\qed
\vskip.2cm

\section{Stability estimate for the wave equation}
\label{sec.stability}
This appendix contains an  estimate for wave equations with sources in $L^\infty_{\text{loc}}(\RR;H^{-1}(\RR^d))$. The weak regularity in $x$ is 
compensated by additional regularity in time.
The residuals in the criminal and the non criminal approximation are of that form. 
For completeness the proof is included.  The systems case is exactly 
analogous.

\begin{proposition}
\label{prop:energyh-1source}
 Suppose that $0<m_1<m_2<\infty$ are real numbers and
   $a, \rho\in L^\infty(\RR^d)$ satisfy
 $
 m_1\le  a,\rho \le  m_2.
 $
 There is a constant $C>0$ depending only on $m_1,m_2$ so that 
 for all  $f\in L^1_{\text{loc}}(\RR;L^2(\mathbb{R}^d))$ 
 and  $g\in L^\infty_{\text{loc}}(\RR;L^2(\RR^d;\RR^d))$ with 
$\partial_t g\in L^1_{\text{loc}}(\RR;L^2(\mathbb{R}^d;\RR^d))$ and $f=g=0$ for $t\leq 0$
the solution of
\begin{equation*}
\big(\rho \,\partial_t^2 \ -\  \div(a\,\grad)\big)\, u \ =\  f\  +\  \div\, g \,,
\qquad
u=0 \quad {\rm for}\quad t \le 0\,,
\end{equation*}
 satisfies for all  $t>0$,
\begin{align}
\label{eq:bob}
\|\nabla_{t,x} u(t)\|_{L^2(\mathbb{R}^d)}
 \leq
C
\Big[
\big\|
f
\big\|_{L^1([0,t]\,;\, L^2(\RR^d)) }
+
\big\|
\partial_t g
\big\|_{L^1([0,t];L^2(\RR^d))}
+ \big\|g\|_{L^\infty([0,t];L^2(\RR^d))}
\Big].
\end{align} 
\end{proposition}
{\bf Proof.} 
Approximating $a,\rho,f,g$ by smooth functions it suffices to prove the estimate
for solutions and right hand side that belong to  $H^s([0,t]\times\RR^d)$ for all $s,t>0$
with a  constant that  depends only on the $m_j$.
Introduce continuous functions
$$
E(t)\ :=\ 
\frac12\int_{\R^d} \left( \rho(x)\,  \big|\partial_t u(t,x)  \big|^2\ +\
a(x)\, \big|  \nabla u(t,x) \big|^2 \right) \,dx
\,,
\quad 
{\rm and}
\quad
M(t)  := \sup_{0\le \ut\le t}  E(\ut).
$$
Testing the equation with $\partial_t u$
yields the standard energy identity
$$
E(t)\ =\ 
\int_0^t
\int_{\RR^d}
\partial_tu
\big(
f + \div g\big)\ dx\,dt\,.
$$
Estimate the first of the two summands
on the right as
\begin{align*}
\Big|
\int_0^t\int_{\RR^d} f\partial_t u \ dxdt\Big|
\le
\big\|
f
\big\|_{L^1([0,t]\,;\, L^2(\RR^d))}\
\big\|
\partial_t u
\big\|_{L^\infty([0,t]\,;\, L^2(\RR^d))}
 \le  C\,  \big\|
f
\big\|_{L^1([0,t]\,;\, L^2(\RR^d))}\,
M(t)^{1/2}
\end{align*}

with a constant depending only on the $m_j$.
For the second summand two integrations by parts yield
\begin{align*}
\int_0^t\int_{\RR^d} \div\, g\ \partial_t u \ dxdt
\ &=\
-
\int_0^t\int_{\RR^d}  g\cdot \grad\,\partial_t u \ dxdt
\cr
\ &=\ 
\int_0^t\int_{\RR^d}  \partial_tg\cdot \grad\, u \ dxdt
\ -\
\int_{\RR^d} g(t,x)\cdot \grad\, u(t,x)\ dx\,.
\end{align*}
Therefore
\begin{align*}
\Big|
\int_0^t\int_{\RR^d} \div g\ \partial_t u \ dxdt
&\Big|
\ \le \
\cr
\big\|
\partial_t g
\big\|_{L^1(]0,t[ ;  L^2(\RR^d)  ) }\,
\big\|
& \grad \, u
\big\|_{L^\infty([0,t]\,;\, L^2(\RR^d) ) }
+
\| g\|_{L^\infty([0,t];L^2(\RR^d ) ) }
\|\grad\, u\|_{L^\infty( [0,t];L^2(\RR^d)) }
\cr
\ &\le \ C\, 
\Big( \big\|
\partial_t g
\big\|_{L^1(]0,t[\,;\, L^2(\RR^d))}
+ \big\|g\|_{L^\infty([0,t];L^2(\RR^d))}
\Big)
M(t)^{1/2}\,.
\end{align*}
Combining yields
\begin{equation}
\label{eq:C2}
E(t)\ \le \ C
\Big(
\big\|
f
\big\|_{L^1([0,t]\,;\, L^2(\RR^d)) }
+
\big\|
\partial_t g
\big\|_{L^1([0,t]\,;\, L^2(\RR^d))}
+ \big\|g\|_{L^\infty([0,t];L^2(\RR^d))}
\Big)
M(t)^{1/2}\,.
\end{equation}
For each $t>0$, choose  $0<\ut\le t$ so that $E(\ut)= M(t)$.   
Estimate \eqref{eq:C2} at time $\ut$ yields
\begin{align*}
M(t)\  &\le
\ C
\Big(
\big\|
f
\big\|_{L^1([0,\ut]\,;\, L^2(\RR^d)) }
+
\big\|
\partial_t g
\big\|_{L^1([0,\ut]\,;\, L^2(\RR^d))}
+ \big\|g\|_{L^\infty([0,\ut];L^2(\RR^d))}
\Big)
M(\ut)^{1/2}
\cr
&\le
C
\Big(
\big\|
f
\big\|_{L^1([0,t]\,;\, L^2(\RR^d)) }
+
\big\|
\partial_t g
\big\|_{L^1([0,t]\,;\, L^2(\RR^d))}
+ \big\|g\|_{L^\infty([0,t];L^2(\RR^d))}
\Big)
M(t)^{1/2}\,.
\end{align*}
If $M(t)\ne 0$, 
dividing by $M(t)^{1/2}$ yields
\eqref{eq:bob}.
  If $M(t)=0$, \eqref{eq:bob} holds with $C=0$.
\hfill
$\Box$
\vskip.3cm

\textbf{Acknowledgments.} {G.A. is a member of the DEFI project at INRIA Saclay Ile-de-France.}


\begin{thebibliography}{20}
\bibitem{AGS}{\sc A. Abdulle, M. Grote, Ch. Stohrer},
{\em Finite element heterogeneous multiscale method for the wave equation: long-time effects,} 
Multiscale Model. Simul. {\bf 12} (2014), pp.1230-1257. 

\bibitem{AP}{\sc A. Abdulle \& T. N. Pouchon},
{\em A priori error analysis of the finite element heterogeneous multiscale method for the wave equation in heterogeneous media over long time},
SIAM J. Numer. Anal., 54(3), 1507-1534 (2016).


\bibitem{AP2}{\sc A. Abdulle \& T. N. Pouchon},
{\em Effective models for the multidimensional wave equation in heterogeneous media over long time and numerical homogenization}, 
Mathematical Models \& Methods In Applied Sciences, vol. 26, num. 14 (2016).

\bibitem{ABV}{\sc G.~Allaire, M.~Briane \& M.~Vanninathan}, 
{\em A comparison between two scale asymptotic expansions and Bloch wave expansions for the homogenization of periodic structures}, 
SEMA Journal 73(3), 237-259 (2016). 

\bibitem{APR}{\sc G. Allaire, M. Palombaro, J. Rauch},
{\em Diffractive Geometric Optics  for  Bloch  Wave Packets},
{Archive Rat. Mech. Anal.}, {\bf 202}, pp.373-426 (2011). 

\bibitem{AY}{\sc G. Allaire, T. Yamada},
{\em Optimization of dispersive coefficients in the homogenization of the wave equation in periodic structures,}
submitted. HAL preprint: hal-01341082 (July 2016).

\bibitem{andrianov}{\sc I. Andrianov, V. Bolshakov, V. Danishevsʹkyy, D. Weichert}, 
{\em Higher order asymptotic homogenization and wave propagation in periodic composite materials,} 
Proc. R. Soc. Lond. Ser. A Math. Phys. Eng. Sci. 464, no. 2093, 1181-1201 (2008).

\bibitem{BP}{\sc  N. Bakhvalov, G. Panasenko},
{\em Homogenization: Averaging Processes in Periodic Media,}
   Kluwer, Dordrecht (1989).

\bibitem{BG}{\sc A. Benoit, A. Gloria},
{\em Long-time homogenization and asymptotic ballistic transport of classical waves,}
arXiv:1701.08600 (2017)

\bibitem{BLP}{\sc A.~Bensoussan, J.-L.~Lions \& G.~Papanicolaou}, 
{\em Asymptotic analysis for periodic structures}, corrected reprint of the 1978 original, AMS Chelsea Publishing, Providence, RI, (2011).

\bibitem{BFM}{\sc S. Brahim-Otsmane, G. Francfort \& F. Murat},
{\em Correctors for the homogenization of the wave and heat equations,}
J. Math. Pures Appl. (9) 71:197--231 (1992).

\bibitem{COV1}{\sc C.~Conca, R.~Orive \& M.~Vanninathan},
{\em Bloch approximation in homogenization and applications},
SIAM J. Math. Anal., {\bf 33} (5) (2002), 1166-1198.

\bibitem{COV2}{\sc C.~Conca, R.~Orive \& M.~Vanninathan},
{\em On Burnett coefficients in periodic media}, 
{ J.~Math. Phys.}, {\bf 47} (3) (2006), 11~pp.

\bibitem{DLS}{\sc T.~Dohnal, A.~Lamacz \& B.~Schweizer},
{\em Bloch-wave homogenization on large time scales and dispersive effective wave equations}, 
{ Multiscale Model. Simul.}, {\bf 12} (2) (2014), 488-513.

\bibitem{fish}{\sc J. Fish, W. Chen, \& G. Nagai},
{\em Non-local dispersive model for wave propagation in heterogeneous media: multi-dimensional case,} 
Internat. J. Numer. Methods Engrg., {\bf 54}(3):347-363, (2002).

\bibitem{JKO}{\sc V Jikov, S. Kozlov, O. Oleinik},
{\em Homogenization of differential operators and integral functionals,} 
Springer, Berlin, (1995). 

\bibitem{Lam}{\sc A.~Lamacz},
{\em Dispersive effective models for waves in heterogeneous media}, 
Math. Models Methods Appl. Sci., {\bf 21} (9) (2011), 1871-1899.

\bibitem{lannes} 
{\sc D.~Lannes},
 {\em High-frequency nonlinear optics: from the nonlinear Schr\"{o}dinger approximation to ultrashort-pulses equations},
  Proc. Roy. Soc. Edinburgh Sect. A 141 (2011), no. 2, 253-286. 

\bibitem{rauch}{\sc J.~Rauch},
{\em Hyperbolic partial differential equations and geometrics optics}, 
volume 133 of {\em Graduate Studies in Mathematics},
American Mathematical Society (2012).

\bibitem{SP}{\sc E. Sanchez-Palencia},
   {\em Non-Homogeneous Media and Vibration Theory},
   Springer Lecture Notes in Physics {\bf 129} (1980). 

\bibitem{santosa}{\sc  F. Santosa, W. Symes},
{\em A dispersive effective medium for wave propagation in periodic composites}, 
{ SIAM J. Appl. Math.}, {\bf 51} (1991), 984-1005.

\bibitem{SmCh}{\sc V. P. Smyshlyaev, K. D. Cherednichenko},
{\em On rigorous derivation of strain gradient effects in the overall behaviour of periodic heterogeneous media}, 
The J. R. Willis 60th anniversary volume. { J. Mech. Phys. Solids}, {\bf 48} (6-7) (2000), 1325-1357.


\end{thebibliography}
\end{document}